\documentclass{article}

\usepackage{arxiv}

\usepackage[utf8]{inputenc} 
\usepackage[T1]{fontenc}    
\usepackage{hyperref}       
\usepackage{url}            
\usepackage{booktabs}       
\usepackage{amsfonts}       
\usepackage{nicefrac}       
\usepackage{microtype}      
\usepackage{lipsum}		
\usepackage{graphicx}
\usepackage{mathtools}
\usepackage{amssymb}
\usepackage{amsthm}
\usepackage{amsmath}
\usepackage{doi}

\usepackage{cases}
\usepackage{microtype}
\usepackage[utf8]{inputenx}
\usepackage{amssymb}
\usepackage{amsthm}
\usepackage{mathtools, cases}
\usepackage{mathbbol}
\usepackage{makecell}
\usepackage{dsfont}
\usepackage{bbm}
\usepackage{mathabx}
\allowdisplaybreaks

\renewcommand{\sinh}{\operatorname{sinh}}

\newtheorem{thm}{Theorem}

\title{Dirichlet problems and exit distributions for the telegraph process and its planar extensions}

\author{ 
	\href{https://orcid.org/0000-0002-6163-044X}{Manfred Marvin Marchione}\\
	Independent Researcher\\
	\texttt{manfredmarvin.marchione@gmail.com}
	\And
	\href{https://orcid.org/0000-0002-6421-533X}{Enzo Orsingher} \\
	Department of Statistical Sciences\\
	Sapienza University of Rome\\
	\texttt{enzo.orsingher@uniroma1.it}
}

\date{\today}

\renewcommand{\headeright}{}
\renewcommand{\undertitle}{}

\begin{document}
\maketitle

\begin{abstract}
In this paper, we study boundary-value problems describing the exit distribution of finite-velocity random motions from prescribed domains. For the standard telegraph process, with and without drift, we derive the Dirichlet problems governing the exit point and mean exit time from a closed interval. We then extend the analysis to a planar finite-velocity model with orthogonal directions, for which we obtain the associated Laplace and Poisson-type equations for the exit distribution and mean exit time. In the special case of an infinite strip, explicit solutions are obtained. In all cases, we show that our equations and results converge, in the hydrodynamic limit, to the corresponding ones for Brownian motion.
\end{abstract}


\section{Introduction}
Since the classical works by Goldstein \cite{goldstein1951} and Kac \cite{kac1974stochastic}, the telegraph process and its multivariate extensions have been extensively studied in the literature. The exact distribution in terms of modified Bessel functions was obtained by Orsingher \cite{orsingher1990}, and the extension to the motion with drift was investigated by Cane \cite{cane1975} and Beghin et al. \cite{beghinnieddu}. In the planar setting, finite-velocity random motions are typically classified according to the number of admissible directions: minimal-direction models were studied by Di Crescenzo \cite{dicrminimal}, while orthogonal models were introduced by Orsingher \cite{orsingher2000exact} and further generalized by Orsingher and Marchione \cite{vortex}. A fundamental property of the telegraph process is that, in the hydrodynamic limit where both the velocity and the switching rate diverge under a suitable scaling, its distribution converges to that of the standard Brownian motion. This holds both for the classical one-dimensional telegraph process (Orsingher \cite{orsingher1990}) and for its planar extensions (see Orsingher and Marchione \cite{vortex}). In view of this result, the telegraph process is often regarded as a finite-velocity counterpart of Brownian motion. Inspired by this crucial intuition, several works in the literature have focused on studying properties of the paths of the telegraph process that mirror the well-known properties of the Brownian motion. Bogachev and Ratanov \cite{bogachev} proved that, similarly to Brownian motion, a Feynman-Kac formula holds for the telegraph process, and they employed it to study the distribution of sojourn times. Di Crescenzo and Pellerey \cite{dicrpellerey} introduced financial models based on the telegraph process, viewed as a finite-velocity analogue of the Brownian motion underlying the Black–Scholes model. The model was further developed by Ratanov \cite{Ratanov}, who added a jump component to the telegraph process in order to rule out arbitrage opportunities. The distribution of first passage times of the telegraph process was investigated by Foong and Kanno \cite{foong}. Cinque and Orsingher \cite{cinquemax} proved that the classical reflection principle for Brownian motion can be extended to the symmetric telegraph process, and subsequently obtained the exact distribution of the maximum of the telegraph process. The distribution of the maximum for the asymmetric telegraph process was later investigated by  Cinque and Orsingher \cite{cinqueasymm}. Pedicone and Orsingher \cite{pedicone} recently obtained the distribution of the telegraphic meander and showed that it converges, in the hydrodynamic limit, to the distribution of the Brownian meander.\\
There exists, however, a fundamental class of problems related to Brownian motion which, to the best of our knowledge, has not yet been extended to finite-velocity random motions in the existing literature. This class of problems concerns the probabilistic interpretation of the Dirichlet problem in terms of the exit of a telegraph process from a given set. In the case of Brownian motion, the probabilistic interpretation of the Dirichlet problem is classical and has been extensively studied within the probabilistic potential theory, with standard references including the monographs by Doob \cite{doob1984classical} and Port and Stone \cite{port2012brownian}. For a probabilistic account of the topic, we refer to the book by Karatzas and Shreve \cite{karatzas2014brownian}. Let $D\subset\mathbb{R}^d$, with $d\ge1$, be an open set and $\{B(t)\}_{t\ge0}$ a $d$-dimensional standard Brownian motion with initial position $x\in D$. Denote by $\tau$ the first exit time of $B(t)$ from $D$, i.e. $\tau=\inf\{t>0:\,B(t)\notin D\}$. For two given functions $f$ and $g$, define the function $u(x)=\mathbb{E}\left[f(B(\tau))+\int_0^\tau g(B(s))\mathrm{d}s\,\lvert B(0)=x\right]$. Under suitable regularity conditions on the functions $f,g$ and the set $D$, the function $u$ is a solution to the boundary-value problem
\begin{equation}\label{intro:DMdirich}
\begin{cases}
\Delta u= -2g\qquad &\mathrm{in}\;D\\
u = f\qquad &\mathrm{on}\;\partial D
\end{cases}
\end{equation}
where $\Delta$ represents the $d$-dimensional Laplacian. Two particular cases of the problem (\ref{intro:DMdirich}) are of special interest from the probabilistic point of view. If $f=0$ and $g=1$, the solution $u$ represents of mean exit time of $B(t)$ from $D$. If $g=0$ and $f(x)=\delta(x-z)$ where $z\in\partial D$ and $\delta(\cdot)$ is the Dirac delta function, then solving (\ref{intro:DMdirich}) yields the distribution of the exit point of $B(t)$ from $D$. In these cases, elegant explicit solutions to the Dirichlet problem (\ref{intro:DMdirich}) are known in the literature for several choices of $D$. These comprise the disk and the ball, for which the classical Poisson kernel arises, as well as the circular annulus and spherical shell, half-planes, infinite strips, and various other domains.\\
The aim of the present paper is to investigate the exit distribution of finite-velocity random motions from a prescribed domain. We focus on both the distribution of the exit point and the mean exit time for the telegraph process with and without drift, as well as its planar extension with orthogonal directions. Our results are obtained by solving suitable linear systems of differential equations with appropriate boundary conditions. In other words, we construct Dirichlet problems whose solutions describe the exit distribution of random motions with finite velocity. Moreover, we obtain the finite-velocity counterparts of the classical Laplace and Poisson equations associated with Brownian motion. Finally, we show that, under the usual hydrodynamic scaling limit, the presented results converge to their classical Brownian motion analogues.\\
We start our analysis by studying the exit distribution of a standard one-dimensional telegraph process from an interval. Let $\{X(t)\}_{t\ge0}$ be a telegraph process with initial position $x\in\mathbb{R}$ and constant velocity $c>0$. At the initial time $t=0$, the process starts moving rightwards or leftwards, each with probability $\frac{1}{2}$, and changes direction according to a homogeneous Poisson process with intensity $\lambda>0$. We denote by $D(t)$ the direction of $X(t)$ at time $t$, and we say that $D(t)=d_0$ when $X(t)$ is moving rightwards, while $D(t)=d_1$ when $X(t)$ is moving leftwards. Consider a closed interval $[a,b]$ such that $x\in[a,b]$, and let $\tau$ be the first exit time of the telegraph process from $[a,b]$, i.e.
$$\tau=\inf\{t>0:\;X(t)\notin[a,b]\}.$$
We study the probability of $X(t)$ exiting the interval $[a,b]$ from the upper endpoint. By conditioning with respect to the initial direction of $X(t)$, we consider, for $j=0,1$, the functions
\begin{equation}u_j(x)=\mathbb{P}\Big(X(\tau)=b\;\Big\lvert X(0)=x,\; D(0)=d_j\Big),\qquad x\in[a,b].\label{introd:uj}\end{equation}
By removing the conditioning with respect to the initial direction $D(0)$ we define
\begin{equation*}u(x)=\mathbb{P}\Big(X(\tau)=b\;\Big\lvert X(0)=x\Big)=\frac{u_0(x)+u_1(x)}{2},\qquad x\in[a,b].\end{equation*}
We prove that, in the open interval $(a,b)$, the functions (\ref{introd:uj}) are a solution to the linear system of differential equations
\begin{equation}\label{introd:ujdsys}
\begin{cases}
\dfrac{\mathrm{d} u_0}{\mathrm{d}x}=-\dfrac{\lambda}{c}\big[u_1-u_0\big]\\[0.75em]
\dfrac{\mathrm{d} u_1}{\mathrm{d}x}=-\dfrac{\lambda}{c}\big[u_1-u_0\big].
\end{cases}
\end{equation}
Moreover, in view of the probabilistic interpretation of the problem, we show that the system (\ref{introd:ujdsys}) is subject to the boundary conditions
\begin{equation}\label{introd:ujbound}\begin{cases}u_0(b)=1\\u_1(a)=0.\end{cases}\end{equation}
To understand the conditions (\ref{introd:ujbound}), it is sufficient to observe that, if the process $X(t)$ has initial position $x=b$ and starts moving rightwards, it immediately exits the interval through $b$ with probability 1. In contrast, if the initial position is $x=a$ and the process starts moving leftwards, then $X(t)$ immediately exits the interval $[a,b]$ through the lower endpoint $a$, and hence cannot exit through $b$. The system (\ref{introd:ujdsys}) together with the boundary conditions (\ref{introd:ujbound}) forms a Dirichlet problem connected with the exit distribution of the telegraph process. By solving this problem, we obtain the distribution of the exit point of $X(t)$ from $[a,b]$, both conditional and unconditional on the initial direction of the process. In particular, in the unconditional case, we prove that 
\begin{equation}\mathbb{P}\Big(X(\tau)=b\;\Big\lvert X(0)=x\Big)=\frac{1}{2}\,\frac{c+2\lambda\,(x-a)}{c+\lambda\,(b-a)}.\label{introd:u_totsol}\end{equation}
We remark that, in the hydrodynamic limit for $\lambda,c\to+\infty$ under the usual scaling $\frac{\lambda}{c^2}\to1$, formula (\ref{introd:u_totsol}) becomes
\begin{equation}\lim_{\substack{\lambda,c\to+\infty\\\lambda/c^2\to1}}\mathbb{P}\Big(X(\tau)=b\;\Big\lvert X(0)=x\Big)=\frac{x-a}{b-a}\label{introd:u_totsollim}\end{equation}
which coincides with the well-known distribution of the exit point from $[a,b]$ of a univariate Brownian motion having initial position $B(0)=x$. This is consistent with the fact that the telegraph process converges in distribution to the Brownian motion in the hydrodynamic limit. We also emphasize that, in view on the system (\ref{introd:ujdsys}), it can be verified that the function $u(x)$ satisfies the differential equation $\Delta u = 0$, where $\Delta$ represents the one-dimensional Laplacian. Hence, the one-dimensional Laplace equation is satisfied both by the exit distribution (\ref{introd:u_totsol}) of the telegraph process and the analogue distribution (\ref{introd:u_totsollim}) for Brownian motion. However, observe that the comparison between formulas (\ref{introd:u_totsol}) and (\ref{introd:u_totsollim}) highlights a fundamental difference between the telegraph process and Brownian motion. In particular, if the telegraph process starts from one endpoint of the interval $[a,b]$ and has initial direction towards the opposite endpoint, then there is a strictly positive probability that it exits the interval for the first time through the opposite endpoint. In the case of Brownian motion, if the process is started at an endpoint of the interval, the it must exit the interval immediately through that endpoint, as a consequence of the oscillating behaviour of Brownian motion.\\
We continue our analysis by studying the mean exit time of the telegraph process $X(t)$ from the interval $[a,b]$. Thus, by conditioning with respect to the initial direction, we consider, for $j=0,1$, the functions
\begin{equation}h_j(x)=\mathbb{E}\left[\tau\,\Big\lvert X(0)=x,\,D(0)=d_j\right],\qquad x\in[a,b].\label{introd:hjdef}\end{equation} and, unconditional with respect to the initial direction,
\begin{equation*}h(x)=\mathbb{E}\left[\tau\,\Big\lvert X(0)=x\right]=\frac{h_0(x)+h_1(x)}{2},\qquad x\in[a,b].\end{equation*}
We prove that, for $x\in(a,b)$, the conditional mean exit times (\ref{introd:hjdef}) are a solution to the inhomogeneous system of differential equations
\begin{equation}\label{introd:hjdsys}
\begin{cases}
\dfrac{\mathrm{d} h_0}{\mathrm{d}x}=-\dfrac{\lambda}{c}\big[h_1-h_0\big]-\dfrac{1}{c}\\[0.75em]
\dfrac{\mathrm{d} h_1}{\mathrm{d}x}=-\dfrac{\lambda}{c}\big[h_1-h_0\big]+\dfrac{1}{c}.
\end{cases}
\end{equation}
subject to the boundary conditions
\begin{equation}\label{introd:hboundcond}
\begin{cases}
h_0(b)=0\\
h_1(a)=0.
\end{cases}
\end{equation} 
The interpretation of the conditions (\ref{introd:hboundcond}) is that, if the telegraph process has initial position $x=b$ and initial direction $d_0$ or it starts at $x=a$ with direction $d_1$, then it exits the interval immediately, which implies that the first exit time is 0 with probability 1. By solving the Dirichlet problem given by the system (\ref{introd:hjdsys}) with boundary conditions (\ref{introd:hboundcond}), we are able to obtain the mean exit time of $X(t)$ from $[a,b]$ conditional and unconditional with respect to the initial direction. In particular, in the unconditional case, we have that 
\begin{equation}\mathbb{E}\left[\tau\,\Big\lvert X(0)=x\right]=\frac{\lambda}{c^2}\,(b-x)(x-a)+\frac{b-a}{2c}.\label{introd:got3}\end{equation}
Observe that, in the hydrodynamic limit for $\lambda,c\to+\infty$ with $\frac{\lambda}{c^2}\to1$, formula (\ref{introd:got3}) reduces to
\begin{equation}\lim_{\substack{\lambda,c\to+\infty\\\lambda/c^2\to1}}\mathbb{E}\left[\tau\,\Big\lvert X(0)=x\right]=(b-x)(x-a)\label{introd:got3limit}\end{equation}
which coincides with the mean exit time of a one-dimensional Brownian motion with initial position $B(0)=x$. Moreover, in view of the system (\ref{introd:hjdsys}), it can be verified that the mean exit time (\ref{introd:got3}) is a solution to the equation $\Delta h=-\frac{2\lambda}{c^2}.$ This equation converges, in the hydrodynamic limit, to $\Delta h=-2$ which coincides with the Poisson equation governing the mean exit time for Brownian motion.
Again, we remark that formulas (\ref{introd:got3}) and (\ref{introd:got3limit}) highlight a crucial difference between the telegraph process and Brownian motion. If the telegraph process starts at one endpoint of the interval $[a,b]$ and initially moves towards the interior, then the exit time is strictly positive. Thus, the mean exit time is strictly positive even at the endpoints of $[a,b]$. This is not the case for Brownian motion, which exits the interval immediately if the initial position coincides with one of the endpoints.\\
Our next step in the analysis of the exit distribution of the telegraph process is to generalize the results presented so far to the case of a telegraph process with drift. By adopting the framework of Beghin et al. \cite{beghinnieddu}, we assume that the telegraph process $X(t)$ has velocity $c_0$ and switching rate $\lambda_0$ when it is moving rightwards, while it has velocity $c_1$ and switching rate $\lambda_1$ when moving leftwards. These assumptions produce a drift effect which arises both by the difference between the velocities and by the asymmetry in the switching rates. Even in the case of the telegraph process with drift, we prove that the exit point distribution and the mean exit time are solutions to suitable boundary-value problems which generalize the systems (\ref{introd:ujdsys}) and (\ref{introd:hjdsys}). By solving such problems, we are able to determine the exit distribution of the telegraph process with drift conditional and unconditional on the initial direction. In particular, in the unconditional case, we prove that the exit point distribution is
\begin{equation}\mathbb{P}\Big(X(\tau)=b\;\Big\lvert X(0)=x\Big)=\frac{1}{2}\;\frac{(\lambda_1c_0+\lambda_0c_1)\,e^{\left(\frac{\lambda_0}{c_0}-\frac{\lambda_1}{c_1}\right)x}-2\lambda_1c_0\,e^{\left(\frac{\lambda_0}{c_0}-\frac{\lambda_1}{c_1}\right)a}}{\lambda_0c_1\,e^{\left(\frac{\lambda_0}{c_0}-\frac{\lambda_1}{c_1}\right)b}-\lambda_1c_0\,e^{\left(\frac{\lambda_0}{c_0}-\frac{\lambda_1}{c_1}\right)a}}\label{introd:u_totsol_drift}\end{equation}
and the mean exit time has the rather involved form
\begin{align}\label{introd:got3drift}\mathbb{E}&\left[\tau\,\Big\lvert X(0)=x\right]=\frac{\lambda_0+\lambda_1}{\lambda_0c_1-\lambda_1c_0}\,(x-a)\nonumber\\
&-\frac{(\lambda_0+\lambda_1)(b-a)}{2(\lambda_0c_1-\lambda_1c_0)}\cdot\frac{(\lambda_0c_1+\lambda_1c_0)\,e^{\left(\frac{\lambda_0}{c_0}-\frac{\lambda_1}{c_1}\right)x}-2\lambda_1c_0\,e^{\left(\frac{\lambda_0}{c_0}-\frac{\lambda_1}{c_1}\right)a}}{\lambda_0c_1\,e^{\left(\frac{\lambda_0}{c_0}-\frac{\lambda_1}{c_1}\right)b}-\lambda_1c_0\,e^{\left(\frac{\lambda_0}{c_0}-\frac{\lambda_1}{c_1}\right)a}}\nonumber\\
&+\frac{c_0+c_1}{2(\lambda_0c_1-\lambda_1c_0)}\cdot\frac{\lambda_0c_1\,e^{\left(\frac{\lambda_0}{c_0}-\frac{\lambda_1}{c_1}\right)b}-(\lambda_0c_1+\lambda_1c_0)e^{\left(\frac{\lambda_0}{c_0}-\frac{\lambda_1}{c_1}\right)x}+\lambda_1c_0\,e^{\left(\frac{\lambda_0}{c_0}-\frac{\lambda_1}{c_1}\right)a}}{\lambda_0c_1e^{\left(\frac{\lambda_0}{c_0}-\frac{\lambda_1}{c_1}\right)b}-\lambda_1c_0e^{\left(\frac{\lambda_0}{c_0}-\frac{\lambda_1}{c_1}\right)a}}.
\end{align}
Moreover, as in the driftless case, we investigate the hydrodynamic limit of the above results. We show that, under a suitable scaling, formulas (\ref{introd:u_totsol_drift}) and (\ref{introd:got3drift}), together with their governing equations, converge to the classical counterparts for Brownian motion with drift.\\
In the final part of the paper, we investigate Dirichlet problems related to the exit distribution of a planar finite-velocity random motion. In particular, we consider the planar random motion with orthogonal directions proposed by Orsingher \cite{orsingher2000exact}. Let $\big(X(t),Y(t)\big)$ be a random motion with initial position $(x,y)\in\mathbb{R}^2$ and constant velocity $c>0$. We assume that the process can move along four orthogonal directions $d_j=\left(\cos\left(\frac{\pi j}{2}\right),\,\sin\left(\frac{\pi j}{2}\right)\right),\; j=0,1,2,3$. At the initial time $t=0$, the process starts moving along one of the four possible directions, each chosen with probability $\frac{1}{4}$, and changes direction  according to a Poisson process $N(t)$ with intensity $\lambda>0$. Each time a Poisson event occurs, the process can turn clockwise, from $d_j$ to $d_{j-1}$, or counterclockwise, from $d_j$ to $d_{j+1}$, each with probability $\frac{1}{2}$. For a given closed set $\Gamma\subset\mathbb{R}^2$ such that $(x,y)\in\Gamma$, we consider a point $\mathbf{z}\in\partial\Gamma$, and we investigate the probability density function of $\big(X(t),Y(t)\big)$ exiting $\Gamma$ through $\mathbf{z}$, that is 
\begin{equation}\label{introdu:bivu}u(x,y\mathord{;}\mathbf{z})\mathrm{d}z=\mathbb{P}\Big(\big(X(\tau),Y(\tau)\big)\in\mathrm{d}\mathbf{z}\Big\lvert\,X(0)=x,\,Y(0)=y\Big).\end{equation} where $\mathrm{d}z$ represents the arc length element of $\mathbf{z}$ along $\partial \Gamma$. By treating the probability density function (\ref{introdu:bivu}) as a function of the starting point $(x,y)$, we prove that it is a solution to the partial differential equation
\begin{equation}\label{introd:extendedlaplacian}\Delta u =\frac{c^2}{\lambda^2}\,\frac{\partial ^4u}{\partial x^2\,\partial y^2}.\end{equation}
Thus, equation (\ref{introd:extendedlaplacian}) represents the finite-velocity analogue of the classical Laplace equation governing the distribution of the exit point of a planar Brownian motion. Since the distribution of $\big(X(t),Y(t)\big)$ converges, in the hydrodynamic limit, to that of a bivariate Brownian motion, it is natural to observe that equation (\ref{introd:extendedlaplacian}) reduces in the same limit to the classical Laplace equation $\Delta u = 0$, where $\Delta$ denotes the bivariate Laplacian. In a similar spirit, we are also able to derive a finite-velocity analogue of the Poisson equation governing the mean exit time of $\big(X(t),Y(t)\big)$ from $\Gamma$. We remark that equation (\ref{introd:extendedlaplacian}) is independent of the geometrical shape of the set $\Gamma$. However, the probabilistic solution related to the exit point distribution does depend on the specific choice of the set $\Gamma$. To illustrate the practical approach for studying the distribution (\ref{introdu:bivu}), we consider the special case in which the set $\Gamma$ coincides with the infinite horizontal strip $\Gamma=\{(x,y)\in\mathbb{R}^2:\;0\le y\le L\}$ for fixed $L>0$. In this case, we are able to obtain the exact probability of the process $\big(X(t),Y(t)\big)$ exiting the strip through the lower boundary, which reads
\begin{equation}\label{introdu:p}\mathbb{P}\Big(Y(\tau)=0\;\Big\lvert X(0)=x,\,Y(0)=y\Big)=\frac{c+\lambda\,(L-y)}{\lambda L+2c},\qquad (x,y)\in\Gamma\end{equation}
and the mean exit time
\begin{equation}\label{introdu:h}\mathbb{E}\Big[\tau\;\Big\lvert\;X(0)=x,\,Y(0)=y\Big]=\frac{\lambda\;y\,(L-y)}{c^2}+\frac{2\lambda L +c}{2\lambda\,c},\qquad (x,y)\in\Gamma\end{equation}
We also derive formulas (\ref{introdu:p}) and (\ref{introdu:h}) conditional on the initial direction of the process $\big(X(t),Y(t)\big)$. As for the exact distribution (\ref{introdu:bivu}) of the exit point, we are able to obtain its Fourier transform with respect to the abscissa $x$ of the starting point, and thus we provide an integral representation for $u(x,y\mathord{;}\mathbf{z})$ as an inverse Fourier transform. Moreover, we perform several numerical experiments that provide a graphical representation of the exit-point distribution and confirm the consistency of the Fourier representation with formula (\ref{introdu:p}).\\
The paper is organized in the following manner. In section 2, we study the Dirichlet problems related to the exit point distribution and the mean exit time for a standard telegraph process in a closed interval. We investigate the limiting cases and we discuss the main differences in comparison with the classical results for Brownian motion. The generalization of these results to the case of a telegraph process with drift is presented in section 3. In section 4, we extend our analysis to the planar case. We determine the Poisson-type equation for the planar motion with orthogonal directions and we investigate the exit distribution of the process from an infinite strip.

\section{Exit point and time of a telegraph process from a closed interval}\label{sec:base}
Let $\{X(t)\}_{t\ge0}$ be a telegraph process with initial position $x\in\mathbb{R}$ and constant velocity $c>0$. At the initial time $t=0$, the random motion $X(t)$ starts moving either rightwards or leftwards with equal probability $\frac{1}{2}$. The direction of $X(t)$ at time $t$ is denoted by $D(t)$, and the right and left directions are denoted respectively by $d_0$ and $d_1$. Clearly, $D(t)\in\{d_0,d_1\}$. The process $X(t)$ is assumed to change direction at Poisson events. We denote by $N(t)$ the homogeneous Poisson process which governs the direction changes and by $\lambda>0$ its constant intensity.\\
\noindent In this section, we are interested in studying the point at which the process $X(t)$ exits from the closed interval $[a,b]$, with $a<b$, under the assumption that $x\in[a,b]$. To this end, we define the exit time $\tau$ as \begin{equation}\label{tau}\tau=\inf\{t>0: X(t)\notin [a,b]\}.\end{equation} Without loss of generality, we consider the probability of $X(t)$ exiting the interval $[a,b]$ from its upper endpoint $b$. Thus, for $j=0,1$, we consider the conditional probabilities
\begin{equation}\label{roccaf}u_j(x)=\mathbb{P}\Big(X(\tau)=b\;\Big\lvert X(0)=x,\; D(0)=d_j\Big),\qquad x\in[a,b].\end{equation}
Clearly, the conditioning with respect to the initial direction $D(0)$ can be removed by defining the function
\begin{equation}\label{u}u(x)=\mathbb{P}\Big(X(\tau)=b\;\Big\lvert X(0)=x\Big)=\frac{u_0(x)+u_1(x)}{2},\qquad x\in[a,b].\end{equation}
The starting point of our analysis is to analyze the behaviour of the telegraph process in a small time interval of amplitude $\Delta t$ to study how the probabilities (\ref{roccaf}) evolve over such a time span. In particular, by analyzing the behaviour of $X(t)$ in the time interval $[0,\,\Delta t]$, it can be verified that the functions $u_j(x),\;j=0,1$, satisfy, for $x\in(a,b)$, the following system of equations:
\begin{equation}\label{ujsys}
\begin{cases}
u_0(x)=u_0(x+c\,\Delta t)\,(1-\lambda\,\Delta t)+u_1(x)\,\lambda\,\Delta t + o(\Delta t)\\
u_1(x)=u_1(x-c\,\Delta t)\,(1-\lambda\,\Delta t)+u_0(x)\,\lambda\,\Delta t + o(\Delta t).
\end{cases}
\end{equation}
To prove the first equation of formula (\ref{ujsys}), note that
\begin{align}u_0(x)=&\mathbb{P}\Big(X(\tau)=b\;\Big\lvert X(0)=x,\; D(0)=d_0\Big)\nonumber\\
=&\sum_{k=0}^{\infty}\mathbb{P}\Big(X(\tau)=b,\,N(\Delta t)=k\;\Big\lvert X(0)=x,\; D(0)=d_0\Big)\nonumber\\
=&\sum_{k=0}^{\infty}\mathbb{P}\Big(X(\tau)=b\;\Big\lvert X(0)=x,\; D(0)=d_0,\;N(\Delta t)=k\Big)\cdot \mathbb{P}\big(N(\Delta t)=k\big)\nonumber\\[0.2em]
=&\mathbb{P}\Big(X(\tau)=b\;\Big\lvert X(\Delta t)=x+c\Delta t,\; D(\Delta t)=d_0\Big)\,(1-\lambda\,\Delta t)\nonumber\\[0.2em]
&\qquad+\mathbb{P}\Big(X(\tau)=b\;\Big\lvert X(\Delta t)=x,\; D(\Delta t)=d_1\Big)\,\lambda\,\Delta t+o(\Delta t)\nonumber\\[0.3em]
=&u_0(x+c\,\Delta t)\,(1-\lambda\,\Delta t)+u_1(x)\,\lambda\,\Delta t + o(\Delta t).\label{markovstep}
\end{align}
\noindent Observe that, in the last step of formula (\ref{markovstep}), we have used the Markov property of the random vector $\big(X(t),D(t)\big)$. The second equation of formula (\ref{ujsys}) can be proved in a similar manner. By performing now a first-order Taylor expansion of the equations in formula (\ref{ujsys}), dividing by $\Delta t$ and taking the limit for $\Delta t\to0$, the following linear system of differential equations is obtained
\begin{equation}\label{ujdsys}
\begin{cases}
\dfrac{\mathrm{d} u_0}{\mathrm{d}x}=-\dfrac{\lambda}{c}\big[u_1-u_0\big]\\[0.75em]
\dfrac{\mathrm{d} u_1}{\mathrm{d}x}=-\dfrac{\lambda}{c}\big[u_1-u_0\big]
\end{cases}
\end{equation}
Moreover, by summing the equations of the system (\ref{ujdsys}), differentiating with respect to $x$ and taking into account formula (\ref{u}), it can be verified that the function $u(x)$ satisfies the ordinary differential equation
\begin{equation}\label{uode}\dfrac{\mathrm{d}^2 u}{\mathrm{d}x^2}=0.\end{equation} Equation (\ref{uode}) implies that the probability for the telegraph process $X(t)$ to exit the interval $[a,b]$ through a given endpoint is a linear function of the initial position $x$. This fundamental property is widely known to hold for the Brownian motion. In particular, let $\{B(t)\}_{t\ge0}$ be a Brownian motion starting at $x\in[a,b]$, and let $\tau$ the exit time of $B(t)$ from $[a,b]$, i.e. formula (\ref{tau}) with $B(t)$ in place of $X(t)$. Then
\begin{equation}\mathbb{P}\Big(B(\tau)=b\;\Big\lvert B(0)=x\Big)=\frac{x-a}{b-a},\qquad x\in[a,b].\label{BMupr}\end{equation}
For details, we refer to the book by Karatzas and Shreve \cite{karatzas2014brownian}. Equation (\ref{uode}) implies that, if the Brownian motion is replaced by a telegraph process, the distribution of the exit point preserves the linear dependence on the starting point. This fact is remarkable, since the telegraph process can be regarded as a finite-velocity counterpart of Brownian motion. Specifically, for $\lambda,c\to+\infty$ under the scaling $\frac{\lambda}{c^2}\to1$, the following limit in distribution holds:
\begin{equation}\lim_{\substack{\lambda,c\to+\infty\\\lambda/c^2\to1}} X(t)\overset{i.d.}{=}B(t).\label{XtoB}\end{equation} We refer to the paper by Orsingher \cite{orsingher1990} for a detailed proof of formula (\ref{XtoB}). However, although the exit-point distribution depends linearly on the starting point for both the telegraph process and Brownian motion, a substantial difference arises in the boundary conditions, as we shall discuss in the present section of this paper.\\
To determine the boundary conditions for the system (\ref{ujdsys}), we start by considering the case in which the starting point $x$ of $X(t)$ lies at the lower endpoint $a$ of $[a,b]$. If the process starts moving leftwards at time $t=0$, it exits the interval immediately through $a$. Hence, the probability that it leaves the interval through $b$ is zero. Conversely, if $x=b$ and the process starts moving rightwards, it exits instantly through $b$ with probability 1. Thus, the following boundary conditions must hold:
\begin{equation}\label{bcep}\begin{cases}u_0(b)=1\\u_1(a)=0.\end{cases}\end{equation}
This permits us to prove the following result.

\begin{thm}\label{thm1}For $x\in[a,b]$, it holds that 
\begin{equation}u_0(x)=\frac{c+\lambda\,(x-a)}{c+\lambda\,(b-a)}\label{u_0sol}\end{equation}
and
\begin{equation}u_1(x)=\frac{\lambda\,(x-a)}{c+\lambda\,(b-a)}.\label{u_1sol}\end{equation}
Consequently,
\begin{equation}u(x)=\frac{1}{2}\,\frac{c+2\lambda\,(x-a)}{c+\lambda\,(b-a)}.\label{u_totsol}\end{equation}
\end{thm}
\begin{proof}To prove formulas (\ref{u_0sol}) and (\ref{u_1sol}), we solve the linear system (\ref{ujdsys}) subject to the boundary conditions (\ref{bcep}). The system (\ref{ujdsys}) can be expressed in matrix form as
\begin{equation}\frac{\mathrm{d}}{\mathrm{d} x}\begin{pmatrix}u_0(x)\\u_1(x)\end{pmatrix}=A\cdot\begin{pmatrix}u_0(x)\\u_1(x)\end{pmatrix}\label{matrixformu}\end{equation} with $A=\begin{pmatrix}\frac{\lambda}{c} & -\frac{\lambda}{c}\\[0.5em]\frac{\lambda}{c} & -\frac{\lambda}{c}\end{pmatrix}.$
 Standard results on linear systems of ordinary differential equations ensure that the general solution to (\ref{matrixformu}) can be expressed in the form 
\begin{equation}\begin{pmatrix}u_0(x)\\u_1(x)\end{pmatrix}=e^{A\,x}\cdot\begin{pmatrix}\kappa_0\\\kappa_1\end{pmatrix}\label{ujgeneralsol}\end{equation} where the constants $\kappa_0,\kappa_1$ depend on the boundary conditions. We recall that the exponential of a square matrix $M$ is defined as $$e^M=I+\sum_{k=1}^{\infty}\frac{M^k}{k!}$$ where $I$ denotes the identity matrix with the same dimension as $M$. Observe now that the matrix $A$ is nilpotent, in the sense that $A^2$ coincides with the null matrix. Hence, it holds that $$e^{A\,x}=I+A\,x=\begin{pmatrix}1+\frac{\lambda x}{c} & -\frac{\lambda x}{c}\\[0.5em]\frac{\lambda x}{c} & 1-\frac{\lambda x}{c}\end{pmatrix}.$$ Therefore, formula (\ref{ujgeneralsol}) yields $$u_0(x)=\kappa_0+(\kappa_0-\kappa_1)\frac{\lambda x}{c},\qquad u_1(x)=\kappa_1+(\kappa_0-\kappa_1)\frac{\lambda x}{c}$$ By employing the boundary conditions (\ref{bcep}), it can be verified that $$\kappa_0=\frac{c-\lambda a}{c+\lambda(b-a)},\qquad \kappa_1=-\frac{\lambda a}{c+\lambda(b-a)}$$ which completes the proof of formulas (\ref{u_0sol}) and (\ref{u_1sol}). Formula (\ref{u_totsol}) follows immediately by using the relationship (\ref{u}).
\end{proof}
In Theorem \ref{thm1}, we have obtained the exact probability of the telegraph process $X(t)$ exiting the interval $[a,b]$ through the upper endpoint $b$ conditional and unconditional on the initial direction of the process. A graphical representation of these probabilities is provided in Figure \ref{fig:u_univ}. Of course, a slight modification of the arguments above permits obtaining the probability of exiting the interval through the lower endpoint $a$. As discussed above, these probabilities are linear functions of the initial position $x$, in analogy with the exit point distribution (\ref{BMupr}) of Brownian motion. However, Theorem \ref{thm1} highlights a crucial difference between the telegraph process and Brownian motion. In the case of the telegraph process, if the starting point $x$ coincides with one of the endpoints of the interval, that is if $x=a$ or $x=b$, the exit of the process from the interval $[a,b]$ may not occur immediately, and the process can exit the interval through the endpoint opposite to its initial position. For instance, assume that the process has initial position $x=a$. If the process starts moving rightwards, it has a positive probability of exiting the interval $[a,b]$ through the upper endpoint $b$. In particular, in view of formula (\ref{u_0sol}), it holds that 
\begin{equation}\mathbb{P}\Big(X(\tau)=b\,\Big\lvert\, X(0)=a,\;D(0)=d_0\Big)=\frac{c}{c+\lambda(b-a)}.\label{remarkab}\end{equation}
Formula (\ref{remarkab}) is remarkable, as it clearly shows how the probability of exiting the interval $[a,b]$ through a given endpoint, when starting from the opposite endpoint, depends on the velocity $c$, the switching intensity $\lambda$ and the width of the interval. This result would be difficult to obtain by direct computation, whereas the Dirichlet problem approach yields it in an elementary way. In contrast to the telegraph process, formula (\ref{BMupr}) shows that for Brownian motion, if the initial position coincides with one of the endpoints of the interval $[a,b]$, then the process exits the interval immediately. This behavior reflects the chaotic nature of Brownian motion, whose paths oscillate rapidly around the starting point. It is interesting to study the behaviour of the probabilities obtained in Theorem \ref{thm1} for large values of $c$ and $\lambda$. Consider the case in which $\lambda$ is fixed and $c\to+\infty$. Clearly, as the velocity $c$ increases, the process $X(t)$ tends to exit the interval $[a,b]$ in smaller time, and the probability of a direction change occurring before the process exiting the interval decreases. Consequently, the process exits the interval through the endpoint in the direction of its initial motion. If the process starts moving with direction $d_0$, it will exit the interval through $b$ in a short time. Similarly, if the initial direction is $d_1$ and the velocity $c$ is large, the process $X(t)$ will likely exit the interval through the lower endpoint $a$. Therefore, it is not surprising that the following limits hold:
$$\lim_{c\to+\infty}u_0(x)=1,\qquad \lim_{c\to+\infty}u_1(x)=0, \qquad \lim_{c\to+\infty}u(x)=\frac{1}{2}.$$
The case in which $c$ is fixed and $\lambda\to+\infty$ is more subtle. For large values of $\lambda$, the telegraph process changes direction with high frequency, which implies that the initial direction of $X(t)$ has little effect on the distribution of the exit point. This is consistent with the fact that
$$\lim_{\lambda\to+\infty}u_0(x)=\lim_{\lambda\to+\infty}u_1(x)=\lim_{\lambda\to+\infty}u(x)=\frac{x-a}{b-a}.$$ However, it must be taken into account that, if the direction changes with too high frequency, the process $X(t)$ is prevented to get too far from its starting position. Indeed, each time the process moves away from the initial position, it immediately changes direction, going back towards the starting point. Intuitively, for $\lambda\to+\infty$ with fixed $c$, the telegraph process $X(t)$ is trapped near its initial position. Later in this section, we shall prove that the mean exit time from the interval $[a,b]$ tends to infinity as $\lambda\to+\infty$.\\
We finally consider the case in which both $\lambda,c\to+\infty$ under the usual scaling $\frac{\lambda}{c^2}\to1$. In this case, it holds that
$$\lim_{\substack{\lambda,c\to+\infty\\\lambda/c^2\to1}}u_0(x)=\lim_{\substack{\lambda,c\to+\infty\\\lambda/c^2\to1}}u_1(x)=\lim_{\substack{\lambda,c\to+\infty\\\lambda/c^2\to1}}u(x)=\frac{x-a}{b-a}$$
which coincides with the exit probability (\ref{BMupr}) for a Brownian motion. This result is consistent with the fact that the telegraph process converges in distribution, in the hydrodynamic limit, to the standard Brownian motion.\\

\begin{figure}[h!]
\centering
\includegraphics[scale=0.24]{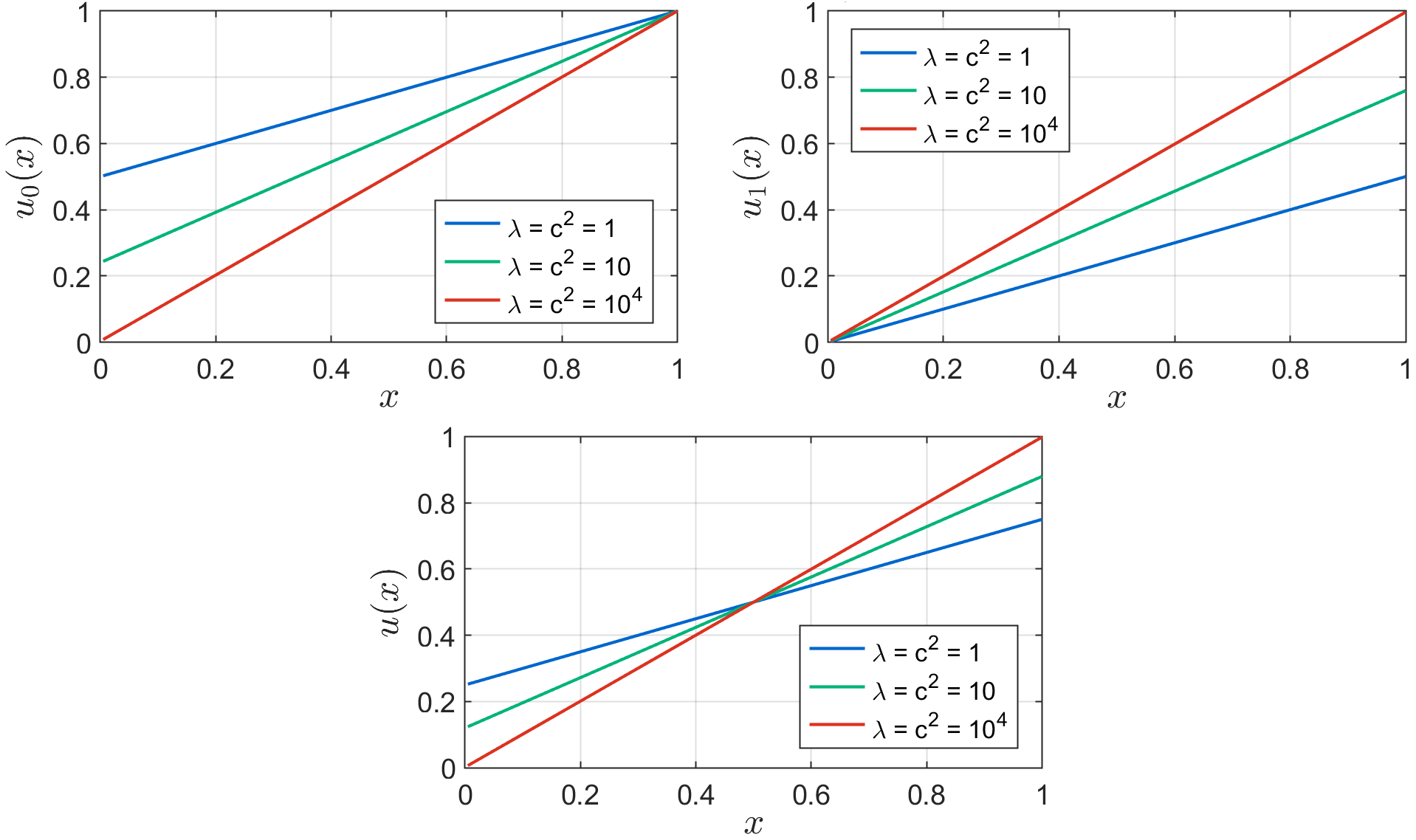}
\caption{Probability of the standard telegraph process exiting the interval $[0,1]$ through the upper endpoint as a function of the starting point $x$, for different values of the velocity $c$ and the switching intensity $\lambda$. The plots are obtained under the parametrization $\lambda=c^2$ for different values of $\lambda$, and they highlight that, when the process starts at an endpoint and initially moves toward the opposite endpoint, there is a positive probability that it exits the interval through the opposite endpoint. However, this probability approaches 0 as $\lambda$ increases. This is consistent with the asymptotic behaviour of the telegraph process, which converges to standard Brownian motion in the hydrodynamic limit.}
\label{fig:u_univ}
\end{figure}

We now continue our analysis by examining the expected exit time of the telegraph process $X(t)$ from the interval $[a,b]$. Consider, for $j=0,1,$ the functions \begin{equation}h_j(x)=\mathbb{E}\left[\tau\,\Big\lvert X(0)=x,\,D(0)=d_j\right],\qquad x\in[a,b].\label{hjdef}\end{equation} Similarly, by removing the conditioning with respect to the initial direction, we define the function
\begin{equation}h(x)=\mathbb{E}\left[\tau\,\Big\lvert X(0)=x\right]=\frac{h_0(x)+h_1(x)}{2},\qquad x\in[a,b].\label{hdef}\end{equation} We claim that, for sufficiently small $\Delta t$, the following system is satisfied for $x\in(a,b)$
\begin{equation}\label{hjsys}
\begin{cases}
h_0(x)=\Delta t+h_0(x+c\,\Delta t)\,(1-\lambda\,\Delta t)+h_1(x)\,\lambda\,\Delta t + o(\Delta t)\\
h_1(x)=\Delta t+h_1(x-c\,\Delta t)\,(1-\lambda\,\Delta t)+h_0(x)\,\lambda\,\Delta t + o(\Delta t).
\end{cases}
\end{equation}
The first equation in system (\ref{hjsys}) is obtained by observing that
\begin{align}
h_0(x)=&\mathbb{E}\left[\tau\,\Big\lvert X(0)=x,\,D(0)=d_0\right]\nonumber\\
=&\sum_{k=0}^{\infty}\mathbb{E}\left[\tau\,\mathds{1}\{N(\Delta t)=k\}\,\Big\lvert X(0)=x,\,D(0)=d_0\right]\nonumber\\
=&\sum_{k=0}^{\infty}\mathbb{E}\left[\tau\,\Big\lvert X(0)=x,\,D(0)=d_0,\,N(\Delta t)=k\right]\cdot \mathbb{P}\big(N(\Delta t)=k\big)\nonumber\\
=&\mathbb{E}\left[\tau\,\Big\lvert X(0)=x,\,D(0)=d_0,\,N(\Delta t)=0\right]\,(1-\lambda\,\Delta t)\nonumber\\
&\qquad+\mathbb{E}\left[\tau\,\Big\lvert X(0)=x,\,D(0)=d_0,\,N(\Delta t)=1\right]\,\lambda\,\Delta t+o(\Delta t)\nonumber\\
=&\mathbb{E}\left[\tau\,\Big\lvert X(\Delta t)=x+c\Delta t,\,D(\Delta t)=d_0\right]\,(1-\lambda\,\Delta t)\nonumber\\
&\qquad+\mathbb{E}\left[\tau\,\Big\lvert X(\Delta t)=x,\,D(\Delta t)=d_1\right]\,\lambda\,\Delta t+o(\Delta t).\label{stepdelic}
\end{align}
Consider now the time shifted process $Z(t)=X(t+\Delta t)$. It is immediate to verify that $Z(t)$ is a telegraph process with the same velocity $c$ and intensity of the direction changes $\lambda$ as the original process $X(t)$. We denote by $D_Z(t)=D(t+\Delta t)$ the direction of $Z(t)$ at time $t$. Moreover, let $\tau_Z$ be the first exit time of $Z(t)$ from the interval $[a,b]$, that is
\begin{equation}\label{tauZ}\tau_Z=\inf\{t>0: Z(t)\notin [a,b]\}.\end{equation}
By comparing formulas (\ref{tau}) and (\ref{tauZ}), and taking into account the definition of infimum of a set, it can be verified that
$$\tau=\tau_Z+\Delta t.$$
Therefore, formula (\ref{stepdelic}) can be expressed in the form
\begin{align}
h_0(x)=&\Delta t+\mathbb{E}\left[\tau_Z\,\Big\lvert Z(0)=x+c\Delta t,\,D_Z(0)=d_0\right]\,(1-\lambda\,\Delta t)\nonumber\\
&\qquad\qquad\qquad+\mathbb{E}\left[\tau_Z\,\Big\lvert Z(0)=x,\,D_Z(0)=d_1\right]\,\lambda\,\Delta t+o(\Delta t).\label{stepdelic2}
\end{align}
Since $Z(t)$ is a telegraph process, formula (\ref{stepdelic2}) coincides with the first equation of the system (\ref{hjsys}). The second equation of the system can be proved by means of a similar procedure. By performing now a first-order Taylor expansion of the equations in formula  (\ref{hjsys}), we obtain the following linear system of ordinary differential equations:
\begin{equation}\label{hjdsys}
\begin{cases}
\dfrac{\mathrm{d} h_0}{\mathrm{d}x}=-\dfrac{\lambda}{c}\big[h_1-h_0\big]-\dfrac{1}{c}\\[0.75em]
\dfrac{\mathrm{d} h_1}{\mathrm{d}x}=-\dfrac{\lambda}{c}\big[h_1-h_0\big]+\dfrac{1}{c}.
\end{cases}
\end{equation}
This implies that, by removing the conditioning with respect to the initial direction, the expected exit time is a solution to the ordinary differential equation
\begin{equation}\label{hode}\dfrac{\mathrm{d}^2 h}{\mathrm{d}x^2}=-\frac{2\lambda}{c^2}.\end{equation}
\noindent In order to solve equations (\ref{hjdsys}) and (\ref{hode}), we need to determine suitable boundary conditions. To this end, observe that, if the process $X(t)$ has initial value $X(0)=b$ and starts moving rightwards, it exits the interval $[a,b]$ immediately. Similarly, the expected exit time is zero if $X(0)=a$ and $D(0)=d_1$. Thus, the following boundary conditions hold:
\begin{equation}\label{hboundcond}
\begin{cases}
h_0(b)=0\\
h_1(a)=0.
\end{cases}
\end{equation} 
\noindent We are now able to prove the following remarkable result.
\begin{thm}\label{thm:fring}
For $x\in[a,b]$, it holds that
\begin{equation}\label{got1}h_0(x)=\frac{\lambda}{c^2}\,(b-x)(x-a)+\frac{b-x}{c}\end{equation}
and
\begin{equation}\label{got2}h_1(x)=\frac{\lambda}{c^2}\,(b-x)(x-a)+\frac{x-a}{c}\end{equation}
Consequently, \begin{equation}h(x)=\frac{\lambda}{c^2}\,(b-x)(x-a)+\frac{b-a}{2c}.\label{got3}\end{equation}
\end{thm}
\begin{proof}
We start by writing equation (\ref{hjdsys}) in matrix form
\begin{equation}\frac{\mathrm{d}}{\mathrm{d} x}\begin{pmatrix}h_0(x)\\[0.3em]h_1(x)\end{pmatrix}=A\cdot\begin{pmatrix}h_0(x)\\[0.3em]h_1(x)\end{pmatrix}+\begin{pmatrix}-\frac{1}{c}\\[0.3em]\;\frac{1}{c}\end{pmatrix}\label{matrixformh}\end{equation} with $A=\begin{pmatrix}\frac{\lambda}{c} & -\frac{\lambda}{c}\\[0.5em]\frac{\lambda}{c} & -\frac{\lambda}{c}\end{pmatrix}$. In order to solve the non-homogeneous system (\ref{matrixformh}) we use the method of variation of constants, which leads to the general solution
\begin{equation}\label{hgeneralsol}
\begin{pmatrix}h_0(x)\\[0.3em]h_1(x)\end{pmatrix}=e^{A\,(x-a)}\cdot\begin{pmatrix}\kappa_0\\[0.3em]\kappa_1\end{pmatrix}+\int_a^x e^{A\,(x-y)}\,\mathrm{d}y\cdot \begin{pmatrix}-\frac{1}{c}\\[0.3em]\;\frac{1}{c}\end{pmatrix}\end{equation} The matrix exponential of $A$ can be computed as in Theorem \ref{thm1}. Thus, formula (\ref{hgeneralsol}) produces the general solution 
$$h_0(x)=\kappa_0+(\lambda\kappa_0-\lambda\kappa_1-1)\,\frac{x-a}{c}-\lambda\left(\frac{x-a}{c}\right)^2$$
and
$$h_1(x)=\kappa_1+(\lambda\kappa_0-\lambda\kappa_1+1)\,\frac{x-a}{c}-\lambda\left(\frac{x-a}{c}\right)^2.$$
By employing the boundary conditions (\ref{hboundcond}) we obtain that
$$\kappa_0=\frac{b-a}{c},\qquad\qquad \kappa_1=0$$ which proves formulas (\ref{got1}) and (\ref{got2}). Formula (\ref{got3}) follows immediately.
\end{proof}
\noindent Having now obtained the explicit expressions of the mean exit time from an interval for the standard telegraph process, it is natural to compare it with the corresponding mean exit time for Brownian motion. While in both cases the mean exit time is a quadratic function of the starting point, Theorem \ref{thm:fring} highlights a crucial difference. In the case of Brownian motion, it is well known that the mean exit time has the form
\begin{equation}\label{zalon}\mathbb{E}\left[\tau\,\Big\lvert B(0)=x\right]=(b-x)(x-a),\qquad x\in[a,b].\end{equation}
Formula (\ref{zalon}) shows that, if a Brownian motion is started at one of the endpoints of the interval $[a,b]$, the mean exit time is 0, since the Brownian motion immediately exits the interval. This behaviour reflects the highly irregular nature of Brownian motion, as discussed above. In contrast, in Theorem \ref{thm:fring} we have proved that, if a telegraph process is started at an endpoint of the interval $[a,b]$, the mean exit time is strictly positive, since the process may initially move towards the opposite endpoint, thus delaying the exit time. We now discuss the limiting cases of the mean exit times presented in Theorem \ref{thm:fring}. We first consider the limit for $c\to+\infty$ with fixed $\lambda$. It is clear that, for high values of $c$, the telegraph process tends to exit the interval $[a,b]$ in a short time, since it reaches the endpoints shortly after the initial time due to its high propagation speed. Indeed, it can be verified that $$\lim_{c\to+\infty}h_0(x)=\lim_{c\to+\infty}h_1(x)=\lim_{c\to+\infty}h(x)=0.$$ As for the limit for $\lambda\to+\infty$ with fixed velocity $c$, the following interesting result holds:
$$\lim_{\lambda\to+\infty}h_0(x)=\lim_{\lambda\to+\infty}h_1(x)=\lim_{\lambda\to+\infty}h(x)=+\infty.$$ The intuitive meaning of this result is that large values of $\lambda$ prevent the telegraph process to get far from its initial position and, for $\lambda\to+\infty$, the process is trapped near its initial position. Therefore, the process cannot exit the interval in the limiting case. We finally discuss the hydrodynamic limit for $\lambda,c\to+\infty$ under the scaling $\frac{\lambda}{c^2}\to1$. In this case, it can be immediately verified that
$$\lim_{\substack{\lambda,c\to+\infty\\\lambda/c^2\to1}}h_0(x)=\lim_{\substack{\lambda,c\to+\infty\\\lambda/c^2\to1}}h_1(x)=\lim_{\substack{\lambda,c\to+\infty\\\lambda/c^2\to1}}h(x)=(b-x)(x-a)$$ which coincides with the mean exit time (\ref{zalon}) for Brownian motion. Of course, this is consistent with the limiting behaviour of the telegraph process. A graphical representation of the mean exit time of the telegraph process from an interval is given in Figure \ref{fig:h_univ}.

\begin{figure}[h!]
\centering
\includegraphics[scale=0.24]{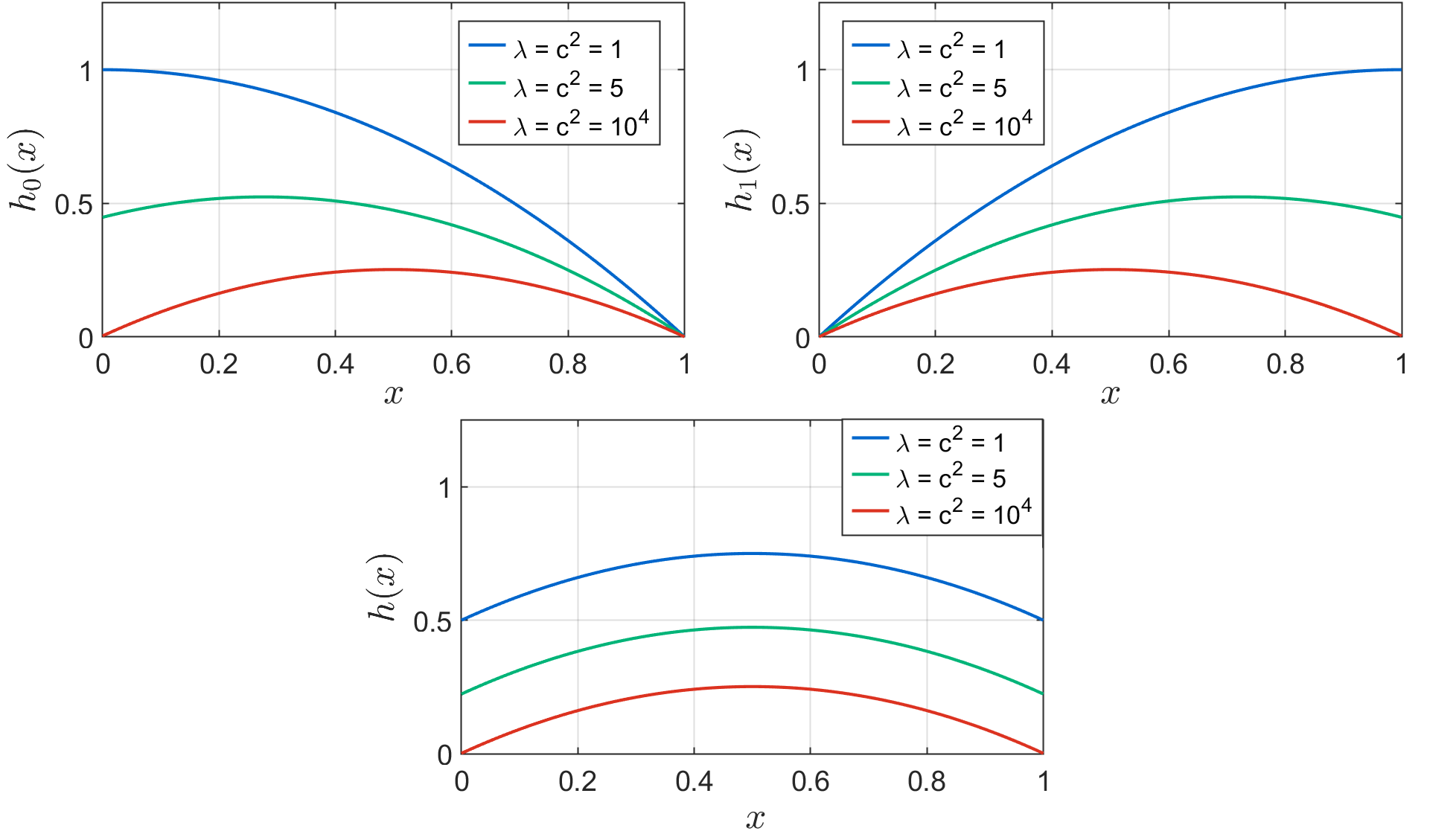}
\caption{Mean exit time from the interval $[0,1]$ for the standard telegraph process as a function of the starting point $x$. The plots were obtained under the parametrization $\lambda=c^2$ for different values of $\lambda$. When the process starts at an endpoint and initially moves toward the opposite endpoint, the mean exit time is strictly positive. However, it tends to 0 as $\lambda$ increases, consistently with the fact that the telegraph process converges to standard Brownian motion in the hydrodynamic limit.}
\label{fig:h_univ}
\end{figure}

\section{Telegraph process with drift in a closed interval}
In this section, we generalize our analysis of the exit distribution to the case of a telegraph process with drift. The telegraph process with drift has been investigated by Cane \cite{cane1975} and Beghin et al. \cite{beghinnieddu}. The authors showed that the process can be transformed into the driftless telegraph process by means of a Lorentz transformation, and subsequently managed to obtain the exact distribution in presence of drift. Moreover, the authors pointed out that there are essentially two ways to include a drift term in the univariate telegraph process. The first way is to assume that the process moves with two distinct velocities depending on whether it is moving leftwards or rightwards. In particular, it can be assumed that the telegraph process $X(t)$ has velocity $c_0$ when it moves rightwards with direction $d_0$, and velocity $c_1$ when it moves leftwards with direction $d_1$. If $c_0\neq c_1$, the difference between the velocities causes the process to drift away from its initial position. The second way to impose a drift to the telegraph process is to assume that the intensity of the direction changes takes two distinct values depending on the direction of the process. Assume that the intensity is $\lambda_0$ when $D(t)=d_0$ and $\lambda_1$ when $D(t)=d_1$. If $\lambda_0\neq \lambda_1$, the time spent moving rightwards and the time spent moving leftwards are different on average, which yields a drift effect. Clearly, these two types of drift can be combined at once. In our analysis, we adopt this general setting.\\
Consider a telegraph process $\{X(t)\}_{t\ge0}$ with initial value $x$. Assume that the process has velocity $c_0$ and switching intensity $\lambda_0$ when it moves with direction $d_0$, while the velocity and the switching rate are $c_1$ and $\lambda_1$ respectively when the direction of $X(t)$ is $d_1$. As in Section \ref{sec:base}, we fix an interval $[a,b]$ such that $x \in [a,b]$, and we denote by $\tau$ the first exit time of $\{X(t)\}_{t \ge 0}$ from $[a,b]$, as defined in formula (\ref{tau}). Without loss of generality, we are interested in studying the probability of the telegraph process exiting the interval $[a,b]$ through the upper endpoint $b$. Thus, for $j=0,1$, we consider the probabilities
$$u_j(x)=\mathbb{P}\Big(X(\tau)=b\;\Big\lvert X(0)=x,\; D(0)=d_j\Big),\qquad x\in[a,b].$$
and, removing the conditioning with respect to the initial direction $D(0)$,
\begin{equation}\label{u_driftdef}u(x)=\mathbb{P}\Big(X(\tau)=b\;\Big\lvert X(0)=x\Big)=\frac{u_0(x)+u_1(x)}{2},\qquad x\in[a,b].\end{equation}
By using the same arguments as in Section \ref{sec:base}, it can be verified that, for $x\in(a,b)$
\begin{equation}\label{ujsys_drift}
\begin{cases}
u_0(x)=u_0(x+c_0\,\Delta t)\,(1-\lambda_0\,\Delta t)+u_1(x)\,\lambda_0\,\Delta t + o(\Delta t)\\
u_1(x)=u_1(x-c_1\,\Delta t)\,(1-\lambda_1\,\Delta t)+u_0(x)\,\lambda_1\,\Delta t + o(\Delta t).
\end{cases}
\end{equation}
Thus, a first-order Taylor expansion of the equations in formula (\ref{ujsys_drift}) yields the linear system of ordinary differential equations
\begin{equation}\label{ujdsys_drift}
\begin{cases}
\dfrac{\mathrm{d} u_0}{\mathrm{d}x}=-\dfrac{\lambda_0}{c_0}\big[u_1-u_0\big]\\[0.75em]
\dfrac{\mathrm{d} u_1}{\mathrm{d}x}=-\dfrac{\lambda_1}{c_1}\big[u_1-u_0\big].
\end{cases}
\end{equation}
It is clear that the boundary conditions (\ref{bcep}) remain valid in the presence of a drift. Moreover, in view of Equation (\ref{ujdsys_drift}), it can be verified that the unconditional distribution of the exit point (\ref{u_driftdef}) satisfies the ordinary differential equation \begin{equation}\label{uode_drift}\dfrac{\mathrm{d}^2 u}{\mathrm{d}x^2}+\left(\frac{\lambda_1}{c_1}-\frac{\lambda_0}{c_0}\right)\dfrac{\mathrm{d} u}{\mathrm{d}x}=0.\end{equation}
Equation (\ref{uode_drift}) shows that, unlike in the driftless case, the exit point distribution is no longer a linear function of the initial state $x$ when the telegraph process $X(t)$ exhibits a drift. In particular, the following result holds.

\begin{thm}\label{thm:driftu}For $x\in[a,b]$, it holds that 
\begin{equation}u_0(x)=\frac{\lambda_0c_1\,e^{\left(\frac{\lambda_0}{c_0}-\frac{\lambda_1}{c_1}\right)x}-\lambda_1c_0\,e^{\left(\frac{\lambda_0}{c_0}-\frac{\lambda_1}{c_1}\right)a}}{\lambda_0c_1\,e^{\left(\frac{\lambda_0}{c_0}-\frac{\lambda_1}{c_1}\right)b}-\lambda_1c_0\,e^{\left(\frac{\lambda_0}{c_0}-\frac{\lambda_1}{c_1}\right)a}}\label{u_0sol_drift}\end{equation}
and
\begin{equation}u_1(x)=\frac{\lambda_1c_0\,e^{\left(\frac{\lambda_0}{c_0}-\frac{\lambda_1}{c_1}\right)x}-\lambda_1c_0\,e^{\left(\frac{\lambda_0}{c_0}-\frac{\lambda_1}{c_1}\right)a}}{\lambda_0c_1\,e^{\left(\frac{\lambda_0}{c_0}-\frac{\lambda_1}{c_1}\right)b}-\lambda_1c_0\,e^{\left(\frac{\lambda_0}{c_0}-\frac{\lambda_1}{c_1}\right)a}}\label{u_1sol_drift}\end{equation}
Consequently,
\begin{equation}u(x)=\frac{1}{2}\;\frac{(\lambda_1c_0+\lambda_0c_1)\,e^{\left(\frac{\lambda_0}{c_0}-\frac{\lambda_1}{c_1}\right)x}-2\lambda_1c_0\,e^{\left(\frac{\lambda_0}{c_0}-\frac{\lambda_1}{c_1}\right)a}}{\lambda_0c_1\,e^{\left(\frac{\lambda_0}{c_0}-\frac{\lambda_1}{c_1}\right)b}-\lambda_1c_0\,e^{\left(\frac{\lambda_0}{c_0}-\frac{\lambda_1}{c_1}\right)a}}\label{u_totsol_drift}\end{equation}
\end{thm}
\begin{proof} Similarly to Theorem \ref{thm1}, we express the solution to the system (\ref{ujdsys_drift}) in matrix form as
\begin{equation}\begin{pmatrix}u_0(x)\\u_1(x)\end{pmatrix}=e^{A\,x}\cdot\begin{pmatrix}\kappa_0\\\kappa_1\end{pmatrix}\label{ujgeneralsol_drift}\end{equation} where
$A=\begin{pmatrix}\frac{\lambda_0}{c_0} & -\frac{\lambda_0}{c_0}\\[0.5em]\frac{\lambda_1}{c_1} & -\frac{\lambda_1}{c_1}\end{pmatrix}$
and the constants $\kappa_0,\kappa_1$ depend on the boundary conditions. We now observe that the matrix $A$ satisfies, for all $k\in\mathbb{N}$, the recurrence relation
$$A^k=\left(\frac{\lambda_0}{c_0}-\frac{\lambda_1}{c_1}\right)^{k-1}A$$ which implies that
\begin{equation}e^{A\, x}=I+\sum_{k=1}^\infty\frac{A^k x^k}{k!}=I+\frac{c_0\,c_1}{\lambda_0 c_1-\lambda_1c_0}\left(e^{\left(\frac{\lambda_0}{c_0}-\frac{\lambda_1}{c_1}\right)x}-1\right)A.\label{recgene}\end{equation} By substituting formula (\ref{recgene}) into (\ref{ujgeneralsol_drift}), we obtain that
\begin{equation}\label{generalsol01_drift}
\begin{aligned}
u_0(x)=&\frac{\lambda_0c_1(\kappa_0-\kappa_1)\,e^{\left(\frac{\lambda_0}{c_0}-\frac{\lambda_1}{c_1}\right)x}+\kappa_1\lambda_0c_1-\kappa_0\lambda_1c_0}{\lambda_0c_1-\lambda_1c_0}\\
u_1(x)=&\frac{\lambda_1c_0(\kappa_0-\kappa_1)\,e^{\left(\frac{\lambda_0}{c_0}-\frac{\lambda_1}{c_1}\right)x}+\kappa_1\lambda_0c_1-\kappa_0\lambda_1c_0}{\lambda_0c_1-\lambda_1c_0}.
\end{aligned}
\end{equation}
To determine the coefficients $\kappa_0$ and $\kappa_1$, we impose the boundary conditions (\ref{bcep}), which yield
\begin{equation}\label{kappa01_drift}
\begin{aligned}
\kappa_0=&\frac{\lambda_0c_1-\lambda_1c_0\,e^{\left(\frac{\lambda_0}{c_0}-\frac{\lambda_1}{c_1}\right)a}}{\lambda_0c_1\,e^{\left(\frac{\lambda_0}{c_0}-\frac{\lambda_1}{c_1}\right)b}-\lambda_1c_0\,e^{\left(\frac{\lambda_0}{c_0}-\frac{\lambda_1}{c_1}\right)a}}\\[0.5em]
\kappa_1=&\frac{\lambda_1c_0\left(1-e^{\left(\frac{\lambda_0}{c_0}-\frac{\lambda_1}{c_1}\right)a}\right)}{\lambda_0c_1\,e^{\left(\frac{\lambda_0}{c_0}-\frac{\lambda_1}{c_1}\right)b}-\lambda_1c_0\,e^{\left(\frac{\lambda_0}{c_0}-\frac{\lambda_1}{c_1}\right)a}}.
\end{aligned}
\end{equation}
The proof is completed by substituting the coefficients given in equation (\ref{kappa01_drift}) into the functions in equation (\ref{generalsol01_drift}). 
\end{proof}
\noindent Theorem \ref{thm:driftu} provides explicit expressions for the exit-point distribution of a telegraph process with drift from an interval, conditional and unconditional on the initial direction. A graphical representation of these probabilities is given in Figure \ref{fig:u_univ_drift}.
\begin{figure}[h!]
\centering
\includegraphics[scale=0.24]{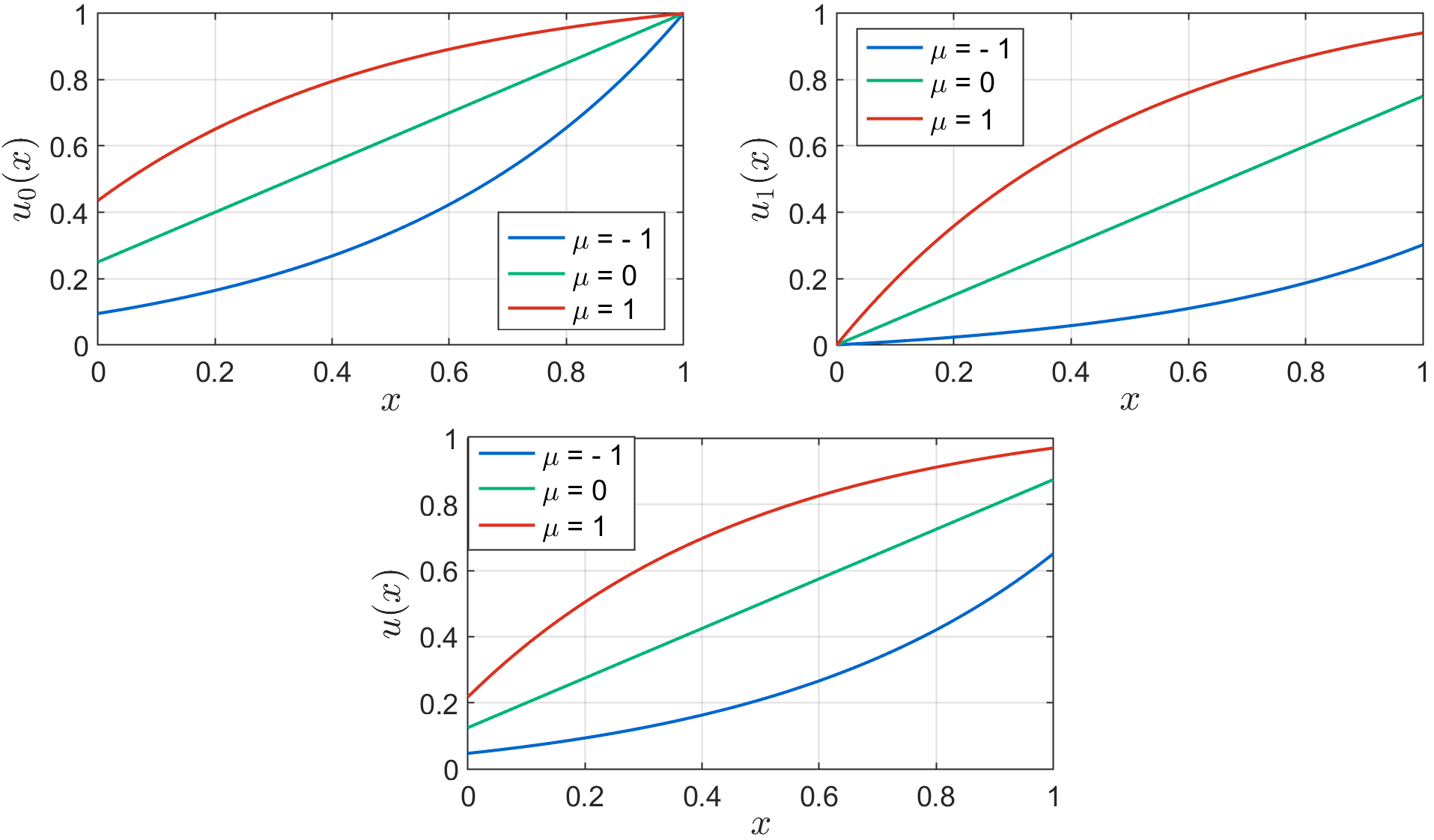}
\caption{Probability of the telegraph process with drift exiting the interval $[0,1]$ through the upper endpoint as a function of the starting point $x$. The plots are obtained under the parametrization $c_0=3$, $\lambda_0=c_0^2$, $\lambda_1=c_1^2$ and $\frac{1}{2}\left(\frac{\lambda_1}{c_1}-\frac{\lambda_0}{c_0}\right)=\mu$, for different values of $\mu$. Under this parametrization, the parameter $\mu$ describes the asymmetry of the process and its sign governs the concavity or convexity of the exit probabilities as functions of the initial position.}
\label{fig:u_univ_drift}
\end{figure}
As in the symmetric telegraph case, our results can be verified to be consistent with the classical exit distribution for Brownian motion with drift. To this aim, however, it is necessary to establish a suitable scaling under which both velocities and switching rates diverge in the hydrodynamic limit. In particular, as $\lambda_0,\lambda_1,c_0,c_1\to+\infty$, we assume that $\frac{\lambda_0}{c_0^2}\to1$, $\frac{\lambda_1}{c_1^2}\to1$ and, for a fixed constant $\mu\in\mathbb{R}$, $\frac{1}{2}\left(\frac{\lambda_1}{c_1}-\frac{\lambda_0}{c_0}\right)\to\mu.$ Under these scaling assumptions, it can be immediately verified that 
$$\lim_{\lambda_0,\lambda_1,c_0,c_1\to+\infty}u(x)=\frac{e^{-2\mu x}-e^{-2\mu a}}{e^{-2\mu b}-e^{-2\mu a}}$$ which coincides with the exit point distribution of a one-dimensional Brownian motion with drift coefficient $\mu$. Clearly, the same limit also holds for the probabilities $u_0$ and $u_1$. Moreover, we remark that, in the hydrodynamic limit, equation (\ref{uode_drift}) reduces to $\Delta u +2\mu u=0$, which is the classical Laplace-type equation which governs the exit distribution of a Brownian motion with drift.\\

We conclude this section by studying the mean exit time of the telegraph process with drift from the interval $[a,b]$. Thus, we retain the assumption that $X(t)$ moves rightwards with velocity $c_0$ and switching intensity $\lambda_0$, and leftwards with velocity $c_1$ and intensity $\lambda_1$. Under this assumption, we consider, for $j=0,1$ as in Section \ref{sec:base}, the functions
\begin{equation}h_j(x)=\mathbb{E}\left[\tau\,\Big\lvert X(0)=x,\,D(0)=d_j\right],\qquad x\in[a,b]\label{hjdefdrift}\end{equation}and
\begin{equation}h(x)=\mathbb{E}\left[\tau\,\Big\lvert X(0)=x\right]=\frac{h_0(x)+h_1(x)}{2},\qquad x\in[a,b].\label{hdefdrift}\end{equation} By means of the usual arguments, it can be verified that the following system of equations holds for $x\in(a,b)$ in the presence of drift
\begin{equation*}
\begin{cases}
h_0(x)=\Delta t+h_0(x+c_0\,\Delta t)\,(1-\lambda_0\,\Delta t)+h_1(x)\,\lambda_0\,\Delta t + o(\Delta t)\\
h_1(x)=\Delta t+h_1(x-c_1\,\Delta t)\,(1-\lambda_1\,\Delta t)+h_0(x)\,\lambda_1\,\Delta t + o(\Delta t)
\end{cases}
\end{equation*}
for a small time interval $\Delta t$. Thus, the conditional expected values in formula (\ref{hjdefdrift}) satisfy the non-homogeneous system of ordinary differential equations
\begin{equation}\label{hjdsysdrift}
\begin{cases}
\dfrac{\mathrm{d} h_0}{\mathrm{d}x}=-\dfrac{\lambda_0}{c_0}\big[h_1-h_0\big]-\dfrac{1}{c_0}\\[0.75em]
\dfrac{\mathrm{d} h_1}{\mathrm{d}x}=-\dfrac{\lambda_1}{c_1}\big[h_1-h_0\big]+\dfrac{1}{c_1}.
\end{cases}
\end{equation}
Moreover, in view of the system (\ref{hjdsysdrift}), the expected value (\ref{hjdefdrift}) satisfies the ordinary differential equation
\begin{equation}\label{hodedrift}\dfrac{\mathrm{d}^2 h}{\mathrm{d}x^2}+\left(\frac{\lambda_1}{c_1}-\frac{\lambda_0}{c_0}\right)\dfrac{\mathrm{d} h}{\mathrm{d}x}=-\frac{\lambda_0+\lambda_1}{c_0c_1}.\end{equation}
Clearly, similarly to the case without drift, the boundary conditions (\ref{hboundcond}) hold. Thus, we obtain the following result.
\begin{thm}\label{thm:put4}
For $x\in[a,b]$, it holds that
\begin{align}\label{got1drift}h_0(x)=\,&\frac{\lambda_0+\lambda_1}{\lambda_0c_1-\lambda_1c_0}\,(x-a)\nonumber\\
&\quad-\frac{(\lambda_0+\lambda_1)(b-a)}{\lambda_0c_1-\lambda_1c_0}\cdot\frac{\lambda_0c_1\,e^{\left(\frac{\lambda_0}{c_0}-\frac{\lambda_1}{c_1}\right)x}-\lambda_1c_0\,e^{\left(\frac{\lambda_0}{c_0}-\frac{\lambda_1}{c_1}\right)a}}{\lambda_0c_1\,e^{\left(\frac{\lambda_0}{c_0}-\frac{\lambda_1}{c_1}\right)b}-\lambda_1c_0\,e^{\left(\frac{\lambda_0}{c_0}-\frac{\lambda_1}{c_1}\right)a}}\nonumber\\
&\quad+\frac{\lambda_0c_1(c_0+c_1)}{\lambda_0c_1-\lambda_1c_0}\cdot\frac{e^{\left(\frac{\lambda_0}{c_0}-\frac{\lambda_1}{c_1}\right)b}-e^{\left(\frac{\lambda_0}{c_0}-\frac{\lambda_1}{c_1}\right)x}}{\lambda_0c_1e^{\left(\frac{\lambda_0}{c_0}-\frac{\lambda_1}{c_1}\right)b}-\lambda_1c_0e^{\left(\frac{\lambda_0}{c_0}-\frac{\lambda_1}{c_1}\right)a}}
\end{align}
and
\begin{align}\label{got2drift}h_1(x)=\,&\frac{\lambda_0+\lambda_1}{\lambda_0c_1-\lambda_1c_0}\,(x-a)\nonumber\\
&\quad-\lambda_1c_0\,\frac{c_0+c_1+(\lambda_0+\lambda_1)(b-a)}{\lambda_0c_1-\lambda_1c_0}\cdot\frac{e^{\left(\frac{\lambda_0}{c_0}-\frac{\lambda_1}{c_1}\right)x}-e^{\left(\frac{\lambda_0}{c_0}-\frac{\lambda_1}{c_1}\right)a}}{\lambda_0c_1e^{\left(\frac{\lambda_0}{c_0}-\frac{\lambda_1}{c_1}\right)b}-\lambda_1c_0e^{\left(\frac{\lambda_0}{c_0}-\frac{\lambda_1}{c_1}\right)a}}
\end{align}
Consequently, 
\begin{align}\label{got3drift}h(&x)=\frac{\lambda_0+\lambda_1}{\lambda_0c_1-\lambda_1c_0}\,(x-a)\nonumber\\
&-\frac{(\lambda_0+\lambda_1)(b-a)}{2(\lambda_0c_1-\lambda_1c_0)}\cdot\frac{(\lambda_0c_1+\lambda_1c_0)\,e^{\left(\frac{\lambda_0}{c_0}-\frac{\lambda_1}{c_1}\right)x}-2\lambda_1c_0\,e^{\left(\frac{\lambda_0}{c_0}-\frac{\lambda_1}{c_1}\right)a}}{\lambda_0c_1\,e^{\left(\frac{\lambda_0}{c_0}-\frac{\lambda_1}{c_1}\right)b}-\lambda_1c_0\,e^{\left(\frac{\lambda_0}{c_0}-\frac{\lambda_1}{c_1}\right)a}}\nonumber\\
&+\frac{c_0+c_1}{2(\lambda_0c_1-\lambda_1c_0)}\cdot\frac{\lambda_0c_1\,e^{\left(\frac{\lambda_0}{c_0}-\frac{\lambda_1}{c_1}\right)b}-(\lambda_0c_1+\lambda_1c_0)e^{\left(\frac{\lambda_0}{c_0}-\frac{\lambda_1}{c_1}\right)x}+\lambda_1c_0\,e^{\left(\frac{\lambda_0}{c_0}-\frac{\lambda_1}{c_1}\right)a}}{\lambda_0c_1e^{\left(\frac{\lambda_0}{c_0}-\frac{\lambda_1}{c_1}\right)b}-\lambda_1c_0e^{\left(\frac{\lambda_0}{c_0}-\frac{\lambda_1}{c_1}\right)a}}
\end{align}
\end{thm}
\begin{proof}
To find the solution to the system (\ref{hjdsysdrift}), we solve the associated homogeneous system and use the method of variation of constants. This leads to the general solution
\begin{equation}\label{targa}
\begin{pmatrix}h_0(x)\\[0.3em]h_1(x)\end{pmatrix}=e^{A\,(x-a)}\cdot\begin{pmatrix}\kappa_0\\[0.3em]\kappa_1\end{pmatrix}+\int_a^x e^{A\,(x-y)}\,\mathrm{d}y\cdot \begin{pmatrix}-\frac{1}{c_0}\\[0.3em]\;\frac{1}{c_1}\end{pmatrix}\end{equation} where $A=\begin{pmatrix}\frac{\lambda_0}{c_0} & -\frac{\lambda_0}{c_0}\\[0.5em]\frac{\lambda_1}{c_1} & -\frac{\lambda_1}{c_1}\end{pmatrix}$. The matrix exponential of $A$ can be obtained as in formula  (\ref{recgene}). Thus, the solution (\ref{targa}) can be expressed as
\begin{align}h_0(x)=&\kappa_0+\frac{\lambda_0+\lambda_1}{\lambda_0c_1-\lambda_1c_0}\,(x-a)\nonumber\\
&\;+\lambda_0c_1\left[\frac{\kappa_0-\kappa_1}{\lambda_0c_1-\lambda_1c_0}-\frac{c_0+c_1}{(\lambda_0c_1-\lambda_1c_0)^2}\right]\left(e^{\left(\frac{\lambda_0}{c_0}-\frac{\lambda_1}{c_1}\right)(x-a)}-1\right)\nonumber\end{align}
and
\begin{align}h_1(x)=&\kappa_1+\frac{\lambda_0+\lambda_1}{\lambda_0c_1-\lambda_1c_0}\,(x-a)\nonumber\\
&\;+\lambda_1c_0\left[\frac{\kappa_0-\kappa_1}{\lambda_0c_1-\lambda_1c_0}-\frac{c_0+c_1}{(\lambda_0c_1-\lambda_1c_0)^2}\right]\left(e^{\left(\frac{\lambda_0}{c_0}-\frac{\lambda_1}{c_1}\right)(x-a)}-1\right).\nonumber\end{align}
By imposing the boundary conditions (\ref{hboundcond}), we obtain that
$$\kappa_0=\frac{(\lambda_0c_1-\lambda_1c_0)(\lambda_0+\lambda_1)(b-a)-\lambda_0c_1(c_0+c_1)\left(e^{\left(\frac{\lambda_0}{c_0}-\frac{\lambda_1}{c_1}\right)(b-a)}-1\right)}{(\lambda_0c_1-\lambda_1c_0)\left(\lambda_1c_0-\lambda_0c_1\,e^{\left(\frac{\lambda_0}{c_0}-\frac{\lambda_1}{c_1}\right)(b-a)}\right)}$$ and $\kappa_1=0$, which yields formulas (\ref{got1drift}) and (\ref{got2drift}). Formula (\ref{got3drift}) follows immediately.
\end{proof}
We have now obtained the explicit expression of the mean exit time of a telegraph process with drift from an interval, and we provide a graphical representation in Figure \ref{fig:h_univ_drift}.
\begin{figure}[h!]
\centering
\includegraphics[scale=0.24]{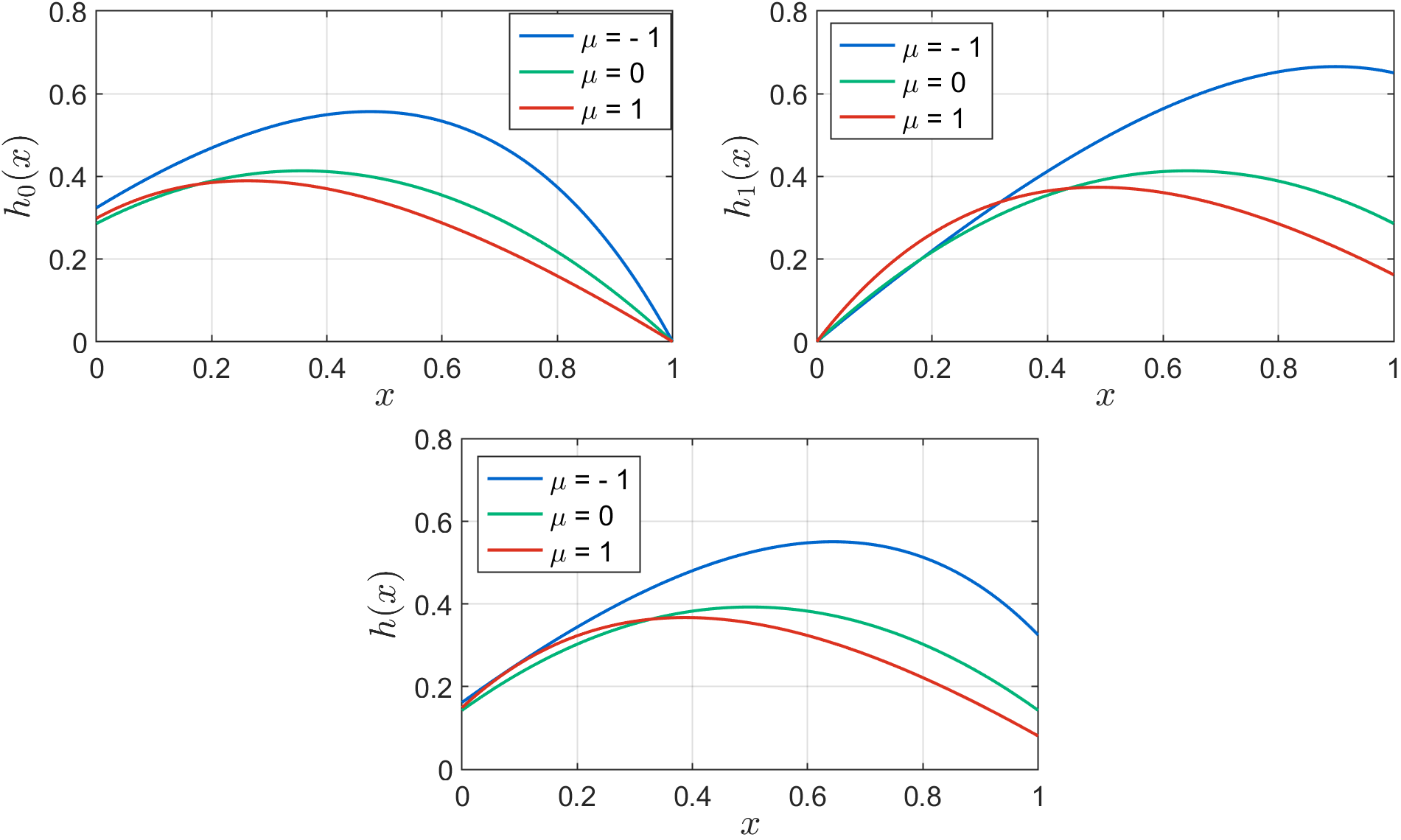}
\caption{Mean exit times from the interval $[0,1]$ for a telegraph process with drift as a function of the initial position $x$. The plots are obtained under the parametrization $c_0=3$, $\lambda_0=c_0^2$, $\lambda_1=c_1^2$ and $\frac{1}{2}\left(\frac{\lambda_1}{c_1}-\frac{\lambda_0}{c_0}\right)=\mu$, for different values of $\mu$. It is worth noting that, unconditional of the initial direction, the mean exit time is no longer a symmetric function of $x$ if $\mu\neq0$.}
\label{fig:h_univ_drift}
\end{figure}
Similarly to the previous results, we shall verify that the results of Theorem \ref{thm:put4} are consistent with the analogue results for Brownian motion with drift. In particular, by taking the hydrodynamic limit for $\lambda_0,\lambda_1,c_0,c_1\to+\infty$ under the scaling $\frac{\lambda_0}{c_0^2}\to1$, $\frac{\lambda_1}{c_1^2}\to1$ and $\frac{1}{2}\left(\frac{\lambda_1}{c_1}-\frac{\lambda_0}{c_0}\right)\to\mu$, it can be verified that 
$$\lim_{\lambda_0,\lambda_1,c_0,c_1\to+\infty}h(x)=\frac{b-a}{\mu}\cdot \frac{e^{-2\mu x}-e^{-2\mu a}}{e^{-2\mu b}-e^{-2\mu a}}-\frac{x-a}{\mu}$$ which coincides with the exit point distribution of a one-dimensional Brownian motion with drift coefficient $\mu$.

\section{On the exit point of a planar random motion from an infinite strip}
Up to now, we have studied how the classical link between the Poisson and Laplace equations and Brownian motion extends to the one-dimensional telegraph process. We now move beyond the univariate case to explore the exit point distribution of planar finite-velocity random motions from a given domain, along with the associated boundary-value problems. To this end, we consider the planar random motion first proposed by Orsingher \cite{orsingher2000exact}. Specifically, we consider a bivariate stochastic process $\big(X(t),Y(t)\big)$ which can move along four orthogonal directions
$$d_j=\left(\cos\left(\frac{\pi j}{2}\right),\,\sin\left(\frac{\pi j}{2}\right)\right),\qquad j=0,1,2,3.$$
Of course, $d_j=d_{j+4n}$ for all integers $n$. At the initial time $t=0$, the process $\big(X(t),Y(t)\big)$ lies at a fixed point $(x,y)\in\mathbb{R}^2$ and starts moving along one of the four possible directions, with each direction having equal probability $\frac{1}{4}$. The process then moves with constant velocity $c>0$ and changes direction at random times. As usual, the direction changes are paced by a homogeneous Poisson process $N(t)$ with intensity $\lambda>0$. Each time a Poisson event occurs, the process can turn clockwise, from $d_j$ to $d_{j-1}$, or counterclockwise, from $d_j$ to $d_{j+1}$, each with probability $\frac{1}{2}$. The direction of $\big(X(t),Y(t)\big)$ at time $t$ is denoted by $D(t)$. Clearly, $D(t)\in\{d_0,d_1,d_2,d_3\}$. For a detailed study of the process $\big(X(t),Y(t)\big)$ and its distribution, we refer to the paper by Marchione and Orsingher \cite{MOcorr}. In the present paper, we are interested in studying the distribution of the exit point of $\big(X(t),Y(t)\big)$ from a given set.\\
\noindent Consider a sufficiently regular connected set $\Gamma\subset{R}^2$, and assume that $(x_0,y_0)\in\Gamma$. Denote by $\tau$ the first exit time of $\big(X(t),Y(t)\big)$ from $\Gamma$, that is \begin{equation}\tau=\{\inf t>0:\;\big(X(t),Y(t)\big)\notin\Gamma\}.\label{bsg}\end{equation}
By considering a point $\mathbf{z}\in\partial \Gamma$, and denoting by $\mathrm{d}z$ its arc length element along $\partial \Gamma$, we are interested in studying the probability density function $u(x,y\mathord{;}\mathbf{z})$ defined as 
\begin{equation}\label{bivu}u(x,y\mathord{;}\mathbf{z})\;\mathrm{d}z=\mathbb{P}\Big(\big(X(\tau),Y(\tau)\big)\in\mathrm{d}\mathbf{z}\Big\lvert\,X(0)=x,\,Y(0)=y\Big).\end{equation}
Similarly, by conditioning with respect to the initial direction of the process, we consider, for $j=0,1,2,3$, the density functions
\begin{equation}\label{bivuj}u_j(x,y\mathord{;}\mathbf{z})\;\mathrm{d}z=\mathbb{P}\Big(\big(X(\tau),Y(\tau)\big)\in\mathrm{d}\mathbf{z}\Big\lvert\,X(0)=x,\,Y(0)=y,\,D(0)=d_j\Big).\end{equation}
Our analysis starts by observing that, by using the same arguments as in the previous sections of this paper, the following system of equations holds for $(x,y)$ in the interior of $\Gamma$ and small values of $\Delta t$
\begin{equation}\label{bivusys}
\begin{cases}
u_0(x,y)=u_0(x+c\,\Delta t,y)\,(1-\lambda\,\Delta t)+u_1(x,y)\,\frac{\lambda\,\Delta t}{2}+u_3(x,y)\,\frac{\lambda\,\Delta t}{2} + o(\Delta t)\\
u_1(x,y)=u_1(x,y+c\,\Delta t)\,(1-\lambda\,\Delta t)+u_0(x,y)\,\frac{\lambda\,\Delta t}{2}+u_2(x,y)\,\frac{\lambda\,\Delta t}{2} + o(\Delta t)\\
u_2(x,y)=u_2(x-c\,\Delta t,y)\,(1-\lambda\,\Delta t)+u_1(x,y)\,\frac{\lambda\,\Delta t}{2}+u_3(x,y)\,\frac{\lambda\,\Delta t}{2} + o(\Delta t)\\
u_3(x,y)=u_3(x,y-c\,\Delta t)\,(1-\lambda\,\Delta t)+u_0(x,y)\,\frac{\lambda\,\Delta t}{2}+u_2(x,y)\,\frac{\lambda\,\Delta t}{2} + o(\Delta t)
\end{cases}\end{equation}
where we have omitted the dependence on $\mathbf{z}$ for simplicity. A first-order expansion of the equations in formula (\ref{bivusys}) yields the linear system of ordinary differential equations
\begin{equation}\label{bivudsys}
\begin{cases}
\dfrac{\partial u_0}{\partial x}=-\dfrac{\lambda}{2c}\,\left[u_1+u_3-2u_0\right]\\[0.75em]
\dfrac{\partial u_1}{\partial y}=-\dfrac{\lambda}{2c}\,\left[u_0+u_2-2u_1\right]\\[0.75em]
\dfrac{\partial u_2}{\partial x}=\dfrac{\lambda}{2c}\,\left[u_1+u_3-2u_2\right]\\[0.75em]
\dfrac{\partial u_3}{\partial y}=\dfrac{\lambda}{2c}\,\left[u_0+u_2-2u_3\right].
\end{cases}\end{equation}
Moreover, the system (\ref{bivudsys}) permits us to recover the partial differential equation which governs the unconditional distribution (\ref{bivu}). By using the standard notation for the Laplace operator, that is $\Delta u=\left(\frac{\partial^2}{\partial x^2}+\frac{\partial^2}{\partial y^2}\right)u$, it can be verified that \begin{equation}\label{extendedlaplacian}\Delta u =\frac{c^2}{\lambda^2}\,\frac{\partial ^4u}{\partial x^2\,\partial y^2}.\end{equation}
We remark that equation (\ref{extendedlaplacian}) is consistent with the standard Laplace equation related to the bivariate Brownian motion. It has been proved by Orsingher and Marchione \cite{vortex} that, in the hydrodynamic limit for $\lambda,c\to+\infty$ with $\frac{\lambda}{c^2}\to1$, the random motion $\big(X(t),Y(t)\big)$ converges in distribution to a bivariate standard Brownian motion with independent components. Thus, it is interesting to observe that the hydrodynamic limit of equation (\ref{extendedlaplacian}) coincides with the Laplace equation $\Delta u=0$, which governs the exit point density for the planar Brownian motion. Hence, equation (\ref{extendedlaplacian}) provides a natural extension of the classical bivariate Laplace equation to the case of finite-velocity random motions with orthogonal directions. Clearly, equation (\ref{extendedlaplacian}) holds regardless of the geometrical structure of the set $\Gamma$. However, the form of its solution does depend on the choice of the domain $\Gamma$, and is therefore difficult to determine in general. To illustrate the principle for finding the solution, we now focus on a specific choice of $\Gamma$, namely an infinite strip.\\

Consider a fixed constant $L>0$ and define the infinite strip
\begin{equation}\label{stripdefinition}\Gamma=\{(x,y)\in\mathbb{R}^2:\;0\le y\le L\}.\end{equation}
Clearly, in the study of the exit distribution of the process $\big(X(t),Y(t)\big)$ from the strip $\Gamma$, it must be taken into account that the exit may occur through either the upper or the lower boundary of the strip. Representative sample paths of the process are illustrated in Figure \ref{fig:sample_paths}. Consequently, one may consider the distributions of the exit locations on both boundaries, as well as the corresponding probabilities of exiting through each boundary and the mean exit time.
\begin{figure}[h!]
\centering
\includegraphics[scale=0.6]{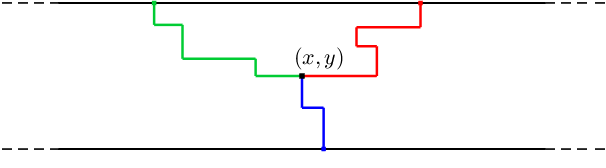}
\caption{Sample paths of the bivariate process $\big(X(t),Y(t)\big)$ with orthogonal directions. The initial position $(x,y)$ is assumed to lie in an horizontal infinite strip. The sample paths of the process and exit the strip through the upper boundary (green and red paths) or through the lower boundary (blue path).}
\label{fig:sample_paths}
\end{figure}
Without loss of generality, we focus our analysis on the exit distribution through the lower boundary of the strip, since the corresponding analysis for the upper boundary is entirely analogous. We start our analysis by investigating the probability density function of the process $\big(X(t),Y(t)\big)$ exiting the set $\Gamma$ through a fixed point $\mathbf{z}\in\partial \Gamma$. We assume that $\mathbf{z}$ belongs to the lower boundary of the strip. In other words, for a fixed $z\in\mathbb{R}$, assume that the exit point admits the representation $\mathbf{z}=(z,0).$ Thus, formulas (\ref{bivu}) and (\ref{bivuj}) can be reformulated respectively as
\begin{equation*}u(x,y\mathord{;}z)\;\mathrm{d}z=\mathbb{P}\Big(X(\tau)\in\mathrm{d}z,\,Y(\tau)=0\,\Big\lvert\,X(0)=x,\,Y(0)=y\Big).\end{equation*}
and, for $j=0,1,2,3$,
\begin{equation}u_j(x,y\mathord{;}z)\;\mathrm{d}z=\mathbb{P}\Big(X(\tau)\in\mathrm{d}z,\,Y(\tau)=0\,\Big\lvert\,X(0)=x,\,Y(0)=y,\,D(0)=d_j\Big).\label{bivuj2}\end{equation}
Under these assumptions, we solve the system (\ref{bivudsys}) by means of a Fourier transform approach. Since the domain of the functions (\ref{bivuj2}) is unbounded in $x$, it seems natural to take the Fourier transform with respect to $x$. Thus, for $j=0,1,2,3$, we define the functions
$$\widetilde{u}_j(\alpha,y\mathord{;}z)=\int_{-\infty}^{\infty}e^{i\alpha x}u_j(x,y\mathord{;}z)\,\mathrm{d}x,\qquad 0<y<L,\;\;\alpha\in\mathbb{R}.$$
Taking the Fourier transform of the equations of the system (\ref{bivudsys}) provides a notable simplification of the problem. The first and third differential equations of the system reduce to the algebraic equations 
\begin{equation}\label{algebrain}
\begin{aligned}
\widetilde{u}_0(\alpha,y\mathord{;}z)=&\frac{\lambda}{2(\lambda+ic\alpha)}\;\big[\widetilde{u}_1(\alpha,y\mathord{;}z)+\widetilde{u}_3(\alpha,y\mathord{;}z)\big]\\
\widetilde{u}_2(\alpha,y\mathord{;}z)=&\frac{\lambda}{2(\lambda-ic\alpha)}\;\big[\widetilde{u}_1(\alpha,y\mathord{;}z)+\widetilde{u}_3(\alpha,y\mathord{;}z)\big].
\end{aligned}
\end{equation}
It is interesting to observe that both identities in equation (\ref{algebrain}) have a clear probabilistic interpretation. To clarify this point, we make some remarks. We recall that the process $\big(X(t),Y(t)\big)$ moves along orthogonal directions, and we are interested in finding the probability density of the process exiting the set $\Gamma$ through the point $\mathbf{z}$. However, in the special case in which $\Gamma$ is an horizontal infinite strip, the process cannot exit from $\Gamma$ while it is moving horizontally, since the set is horizontally unbounded. Thus, in order to exit the strip, the process must change direction and move vertically. For instance, consider the case in which $\big(X(t),Y(t)\big)$ starts moving horizontally with direction $d_0$. Denote by $T_1$ the time at which the first change of direction occurs. Clearly, since the direction changes are paced by a Poisson process, the random variable $T_1$ has exponential distribution with probability density function 
$$f(s)=\lambda\,e^{-\lambda s},\qquad s>0.$$
 Thus, by conditioning on the value of $T_1$, we can write that
\begin{align}u_0(x,y\mathord{;}z)\;\mathrm{d}z=\mathbb{E}\Big[\mathbb{P}\Big(&X(\tau)\in\mathrm{d}z,\,Y(\tau)=0,\nonumber\\
&D(T_1)\in\{d_1,d_3\}\Big\lvert\,X(0)=x,\,Y(0)=y,\,D(0)=d_0,\,T_1\Big)\Big]\label{sonno}
\end{align}
where the expectation is taken with respect to $T_1$. Observe that the event $D(T_1)\in\{d_1,d_3\}$ conditional on $D(0)=d_0$ is a certain event and can be expressed as the union of the incompatible events $D(T_1)=d_1$ and $D(T_1)=d_3$, which both occur with probability $\frac{1}{2}$ conditional on $D(0)=d_0$ . Thus, by using the Markov property of the process $\big(X(t),Y(t),D(t)\big)$, formula (\ref{sonno}) can be expressed as
\begin{equation}u_0(x,y\mathord{;}z)=\frac{1}{2}\mathbb{E}\Big[u_1(x+cT_1,y\mathord{;}z)+u_3(x+cT_1,y\mathord{;}z)\Big]\label{ltprobint}\end{equation}
By taking the Fourier transform of both sides of equation (\ref{ltprobint}) and writing the expected value in integral form, we obtain that
\begin{align}\widetilde{u}_0(\alpha,y\mathord{;}z)=&\frac{\lambda}{2}\int_{-\infty}^{+\infty}\int_0^{+\infty}e^{i\alpha x-\lambda s}\Big[u_1(x+cs,y\mathord{;}z)+u_3(x+cs,y\mathord{;}z)\Big]\,\mathrm{d}s\,\mathrm{d}x\nonumber\\
=&\frac{\lambda}{2}\int_{-\infty}^{+\infty}\int_0^{+\infty}e^{i\alpha w-(i\alpha c+\lambda) s}\Big[u_1(w,y\mathord{;}z)+u_3(w,y\mathord{;}z)\Big]\,\mathrm{d}s\,\mathrm{d}w\nonumber\\
=&\frac{1}{2}\,\mathbb{E}\left[e^{-i\alpha c T_1}\right]\Big[\widetilde{u}_1(\alpha,y\mathord{;}z)+\widetilde{u}_3(\alpha,y\mathord{;}z)\Big].\label{natale}\end{align}
By recalling the characteristic function of the exponential distribution $$\mathbb{E}\left[e^{-i\alpha c T_1}\right]=\frac{\lambda}{\lambda+ic\alpha}$$
formula (\ref{natale}) immediately provides a probabilistic interpretation to the first identity of equation (\ref{algebrain}). Clearly, a similar interpretation holds for the second identity.\\
Resuming the analysis of the system (\ref{bivudsys}), taking the Fourier transform of the second and fourth equations yields
\begin{equation}\label{bivFouriersys}
\begin{cases}
\dfrac{\partial \widetilde{u}_1}{\partial y}=-\dfrac{\lambda}{2c}\,\left[\widetilde{u}_0+\widetilde{u}_2-2\widetilde{u}_1\right]\\[0.75em]
\dfrac{\partial \widetilde{u}_3}{\partial y}=\dfrac{\lambda}{2c}\,\left[\widetilde{u}_0+\widetilde{u}_2-2\widetilde{u}_3\right].
\end{cases}
\end{equation}
By substituting now the identities in equation (\ref{algebrain}) into the system (\ref{bivFouriersys}) we obtain that
\begin{equation}\label{bivFouriersys2}
\begin{cases}
\dfrac{\partial \widetilde{u}_1}{\partial y}=\dfrac{\lambda^3+2\lambda c^2\alpha^2}{2\,c\,(\lambda^2+c^2\alpha^2)}\;\widetilde{u}_1-\dfrac{\lambda^3}{2\,c\,(\lambda^2+c^2\alpha^2)}\;\widetilde{u}_3\\[0.85em]
\dfrac{\partial \widetilde{u}_3}{\partial y}=\dfrac{\lambda^3}{2\,c\,(\lambda^2+c^2\alpha^2)}\;\widetilde{u}_1-\dfrac{\lambda^3+2\lambda c^2\alpha^2}{2\,c\,(\lambda^2+c^2\alpha^2)}\;\widetilde{u}_3.
\end{cases}
\end{equation}
Formula (\ref{bivFouriersys2}) shows that the Fourier transform approach reduces the fourth-order problem (\ref{bivudsys}) to a second-order problem. However, to solve the system (\ref{bivFouriersys2}), we need to impose suitable boundary conditions. To this aim, it is sufficient to take into account the probabilistic interpretation of the problem. Assume that, at time $t=0$, the starting point of the process $\big(X(t),Y(t)\big)$ lies on the upper boundary of the infinite strip $\Gamma$ defined in equation (\ref{stripdefinition}). In other words, assume that the initial position has ordinate $y=L$. If the process starts moving upwards, that is if $D(0)=d_1$, it immediately exits the strip from the upper boundary. Hence, if $y=L$ and $D(0)=d_1$, the probability of $\big(X(t),Y(t)\big)$ exiting $\Gamma$ from below is zero, which implies that \begin{equation}u_1(x,L\mathord{;}z)=0,\qquad \forall x\in\mathbb{R}.\label{ubouu}\end{equation}
Similarly, if the process $\big(X(t),Y(t)\big)$ has initial position on the lower boundary of $\Gamma$ and initial direction $D(0)=d_3$, it immediately exits the strip $\Gamma$ through the lower boundary. In particular, the probability density of exiting through the point $(z,0)$ given that the initial position is $(x,0)$ is
\begin{equation}\label{uboud}u_3(x,0\mathord{;}z)=\delta(z-x),\qquad\forall x\in\mathbb{R}.\end{equation}
By taking the Fourier transform of formulas (\ref{ubouu}) and (\ref{uboud}), we have that
\begin{equation}\label{stripboundarycon}\widetilde{u}_1(\alpha,L\mathord{;}z)=0,\qquad \widetilde{u}_3(\alpha,0\mathord{;}z)=e^{i\alpha z}.\end{equation}
We are now able to prove the following result.

\begin{thm}\label{tildeuthm}
For $j=1,3$, the Fourier transforms $\widetilde{u}_j(\alpha,y\mathord{;}z)$ read
\begin{equation}\label{uhat1solu}\widetilde{u}_1(\alpha,y\mathord{;}z)=\frac{2\lambda^2\sinh\left(\frac{\lvert\alpha\lvert\lambda(L-y)}{\sqrt{\lambda^2+\alpha^2c^2}}\right)e^{i\alpha z}}{\big(\sqrt{\lambda^2+\alpha^2c^2}+\lvert\alpha\lvert c\big)^2\;e^{\frac{\lvert\alpha\lvert\,\lambda\,L}{\sqrt{\lambda^2+\alpha^2c^2}}}-\big(\sqrt{\lambda^2+\alpha^2c^2}-\lvert\alpha\lvert c\big)^2\;e^{-\frac{\lvert\alpha\lvert\,\lambda\,L}{\sqrt{\lambda^2+\alpha^2c^2}}}}\end{equation}
and
\begin{align}\widetilde{u}_3(\alpha,&\,y\mathord{;}z)=\nonumber\\
&e^{i\alpha z}\,\frac{\big(\sqrt{\lambda^2+\alpha^2c^2}+\lvert\alpha\lvert c\big)^2\;e^{\frac{\lvert\alpha\lvert\,\lambda\,(L-y)}{\sqrt{\lambda^2+\alpha^2c^2}}}-\big(\sqrt{\lambda^2+\alpha^2c^2}-\lvert\alpha\lvert c\big)^2\;e^{-\frac{\lvert\alpha\lvert\,\lambda\,(L-y)}{\sqrt{\lambda^2+\alpha^2c^2}}}}{\big(\sqrt{\lambda^2+\alpha^2c^2}+\lvert\alpha\lvert c\big)^2\;e^{\frac{\lvert\alpha\lvert\,\lambda\,L}{\sqrt{\lambda^2+\alpha^2c^2}}}-\big(\sqrt{\lambda^2+\alpha^2c^2}-\lvert\alpha\lvert c\big)^2\;e^{-\frac{\lvert\alpha\lvert\,\lambda\,L}{\sqrt{\lambda^2+\alpha^2c^2}}}}.\label{uhat3solu}\end{align}
For $j=0,2$, the Fourier transforms $\widetilde{u}_0(\alpha,y\mathord{;}z)$ and $\widetilde{u}_2(\alpha,y\mathord{;}z)$ follow immediately by the identities in equation (\ref{algebrain}).\end{thm}
\begin{proof}
We start by expressing the system (\ref{bivFouriersys2}) in matrix form. With the notation
$$A=\begin{pmatrix}
\dfrac{\lambda^3+2\lambda c^2\alpha^2}{2\,c\,(\lambda^2+c^2\alpha^2)}\; & \; -\dfrac{\lambda^3}{2\,c\,(\lambda^2+c^2\alpha^2)}\\
\dfrac{\lambda^3}{2\,c\,(\lambda^2+c^2\alpha^2)}\; & \; -\dfrac{\lambda^3+2\lambda c^2\alpha^2}{2\,c\,(\lambda^2+c^2\alpha^2)}
\end{pmatrix}$$
we can write that
\begin{equation}\label{smallvi}\frac{\partial}{\partial y}\begin{pmatrix}\widetilde{u}_1(\alpha,y\mathord{;}z)\\[0.5em] \widetilde{u}_3(\alpha,y\mathord{;}z)\end{pmatrix}=A\cdot\begin{pmatrix}\widetilde{u}_1(\alpha,y\mathord{;}z)\\[0.5em] \widetilde{u}_3(\alpha,y\mathord{;}z)\end{pmatrix}.\end{equation}
By using standard theory of linear systems of ordinary differential equations, the general solution to the system (\ref{smallvi}) can be represented in terms of the eigenvalues and eigenvectors of the matrix $A$. In particular, denote by $\theta_j,\;j=0,1,$ the eigenvalues of $A$, and let $\mathbf{v}_j,\;j=0,1,$ be the corresponding eigenvectors. We can write that 
$$\begin{pmatrix}\widetilde{u}_1(\alpha,y\mathord{;}z)\\[0.5em] \widetilde{u}_3(\alpha,y\mathord{;}z)\end{pmatrix}=\sum_{j=0}^1k_j\,e^{\theta_j y}\mathbf{v}_j$$ where $k_j,\;j=0,1$, are arbitrary constants. Straightforward calculations permit us to prove that
$$\theta_0=\frac{\lvert\alpha\lvert\,\lambda}{\sqrt{\lambda^2+\alpha^2c^2}},\qquad \theta_1=-\frac{\lvert\alpha\lvert\,\lambda}{\sqrt{\lambda^2+\alpha^2c^2}}.$$ The corresponding eigenvectors are
$$\mathbf{v}_0=\begin{pmatrix}\lvert\alpha\lvert c+\sqrt{\lambda^2+\alpha^2c^2}\\[0.75em]-\lvert\alpha\lvert c+\sqrt{\lambda^2+\alpha^2c^2}\end{pmatrix},\qquad \mathbf{v}_1=\begin{pmatrix}\lvert\alpha\lvert c-\sqrt{\lambda^2+\alpha^2c^2}\\[0.75em]-\lvert\alpha\lvert c-\sqrt{\lambda^2+\alpha^2c^2}\end{pmatrix}.$$
Thus, the general solution to the system (\ref{bivFouriersys2}) reads
$$\widetilde{u}_1(\alpha,y\mathord{;}z)=k_0\,(\sqrt{\lambda^2+\alpha^2c^2}+\lvert\alpha\lvert c)\;e^{\frac{\lvert\alpha\lvert\,\lambda\,y}{\sqrt{\lambda^2+\alpha^2c^2}}}-k_1\,(\sqrt{\lambda^2+\alpha^2c^2}-\lvert\alpha\lvert c)\;e^{-\frac{\lvert\alpha\lvert\,\lambda\,y}{\sqrt{\lambda^2+\alpha^2c^2}}}$$
and
$$\widetilde{u}_3(\alpha,y\mathord{;}z)=k_0\,(\sqrt{\lambda^2+\alpha^2c^2}-\lvert\alpha\lvert c)\;e^{\frac{\lvert\alpha\lvert\,\lambda\,y}{\sqrt{\lambda^2+\alpha^2c^2}}}-k_1\,(\sqrt{\lambda^2+\alpha^2c^2}+\lvert\alpha\lvert c)\;e^{-\frac{\lvert\alpha\lvert\,\lambda\,y}{\sqrt{\lambda^2+\alpha^2c^2}}}$$
By using now the boundary conditions (\ref{stripboundarycon}), it can be verified that
$$k_0=\frac{\big(\sqrt{\lambda^2+\alpha^2c^2}-\lvert\alpha\lvert c\big)\;e^{i\alpha z-\frac{\lvert\alpha\lvert\lambda L}{\sqrt{\lambda^2+\alpha^2c^2}}}}{\big(\sqrt{\lambda^2+\alpha^2c^2}-\lvert\alpha\lvert c\big)^2\;e^{-\frac{\lvert\alpha\lvert\,\lambda\,L}{\sqrt{\lambda^2+\alpha^2c^2}}}-\big(\sqrt{\lambda^2+\alpha^2c^2}+\lvert\alpha\lvert c\big)^2\;e^{\frac{\lvert\alpha\lvert\,\lambda\,L}{\sqrt{\lambda^2+\alpha^2c^2}}}}$$
and
$$k_1=\frac{\big(\sqrt{\lambda^2+\alpha^2c^2}+\lvert\alpha\lvert c\big)\;e^{i\alpha z+\frac{\lvert\alpha\lvert\lambda L}{\sqrt{\lambda^2+\alpha^2c^2}}}}{\big(\sqrt{\lambda^2+\alpha^2c^2}-\lvert\alpha\lvert c\big)^2\;e^{-\frac{\lvert\alpha\lvert\,\lambda\,L}{\sqrt{\lambda^2+\alpha^2c^2}}}-\big(\sqrt{\lambda^2+\alpha^2c^2}+\lvert\alpha\lvert c\big)^2\;e^{\frac{\lvert\alpha\lvert\,\lambda\,L}{\sqrt{\lambda^2+\alpha^2c^2}}}}.$$
Formulas (\ref{uhat1solu}) and (\ref{uhat3solu}) follow immediately.
\end{proof}
We now make some remarks. In principle, the Fourier transforms obtained in Theorem \ref{tildeuthm} can be inverted in order to find the exact distribution of the exit point of the process $\big(X(t),Y(t)\big)$ through the lower boundary of the strip $\Gamma$. In particular, for $j=0,1,2,3$, the probability density functions (\ref{bivuj2}) can be represented in the form
\begin{equation}u_j(x,y\mathord{;}z)=\frac{1}{2\pi}\int_{-\infty}^{\infty}e^{-i\alpha x}\;\widetilde{u}_j(\alpha,y\mathord{;}z)\,\mathrm{d}\alpha,\qquad 0<y<L,\;\;x\in\mathbb{R}.\label{smville}\end{equation}
However, the special case $j=3$, that is the case in which the process starts moving downwards at $t=0$, deserves a specific analysis, as the Fourier integral representation is subtle in this case. In particular, since the Fourier transform (\ref{uhat3solu}) does not vanish for $\lvert\alpha\lvert\to+\infty$, the integral (\ref{smville}) is not convergent in the usual sense for $j=3$. This fact has a fundamental probabilistic interpretation. In particular, as we will now prove, the distribution of the exit point from the strip $L$ has a singular component.\\
\noindent To study the convergence of the inverse Fourier transform (\ref{smville}), we examine the asymptotic behaviour of the functions $\widetilde{u}_j(\alpha,\,y\mathord{;}z),\;j=0,1,2,3,$ for $\lvert\alpha\lvert\to+\infty$. In view of formulas (\ref{uhat1solu}) and (\ref{uhat3solu}), it is immediate to verify that, for large values of $\lvert\alpha\lvert$, the following asymptotic behaviour holds:
$$\widetilde{u}_1(\alpha,\,y\mathord{;}z)\sim\frac{\lambda^2\sinh\left(\frac{\lambda\,(L-y)}{c}\right)\,e^{i\alpha z}}{2\alpha^2 c^2 e^{\frac{\lambda L}{c}}}$$
and
\begin{equation}\widetilde{u}_3(\alpha,\,y\mathord{;}z)\sim e^{-\frac{\lambda y}{c}+i\alpha z}.\label{asym3}\end{equation}
Moreover, equation (\ref{algebrain}) implies that, for large $\lvert\alpha\lvert$,
$$\widetilde{u}_0(\alpha,\,y\mathord{;}z)\sim\frac{\lambda}{2(\lambda+ic\alpha)}\left[\frac{\lambda^2\sinh\left(\frac{\lambda\,(L-y)}{c}\right)}{2\alpha^2 c^2 e^{\frac{\lambda L}{c}}}+e^{-\frac{\lambda y}{c}}\right]\,e^{i\alpha z}$$
$$\widetilde{u}_2(\alpha,\,y\mathord{;}z)\sim\frac{\lambda}{2(\lambda-ic\alpha)}\left[\frac{\lambda^2\sinh\left(\frac{\lambda\,(L-y)}{c}\right)}{2\alpha^2 c^2 e^{\frac{\lambda L}{c}}}+e^{-\frac{\lambda y}{c}}\right]\,e^{i\alpha z}.$$
Thus, it is clear that, for $j=0,1,2$, the integral in formula (\ref{smville}) is well-defined, since the integrand functions decay sufficiently rapidly at infinity. In particular, for $j=0,2$, the convergence of the integral can be proved by using the Dirichlet's test for improper integrals. As for the case $j=3$, the asymptotic behaviour of $\widetilde{u}_3(\alpha,\,y\mathord{;}z)$ clearly shows that such integral does not converge in the usual sense. However, the inverse Fourier transform is meaningful from the probabilistic point of view and can be interpreted in the distributional sense. To clarify this point, we make some crucial remarks.\\
Assume that, at time $t=0$, the process $\big(X(t),Y(t)\big)$ lies at the point $(x,y)$ with $0<y<L$, and starts moving downwards with direction $d_3$. For a fixed point $z\in\mathbb{R}$, we recall that the function $u_3(x,y\mathord{;}z)$ represents the probability density of exiting the strip $\Gamma$ through the point $(z,0)$. Consider now the special case in which $z=x$, that is the case in which the starting point is precisely above the exit point. In this case, it can be verified that the process $\big(X(t),Y(t)\big)$ exits the strip through $(z,0)$ with a strictly positive probability. In particular, recalling the notation (\ref{bsg}) for the first exit time $\tau$, if $z=x$ and the initial direction is $d_3$, the event $X(\tau)=z$ occurs if no changes of direction occur for a sufficiently long time. Specifically, the time needed for the process to reach the lower boundary of the strip is $t^*=\frac{y}{c}$. Hence, the probability of no changes of direction occurring up to time $t^*$ is $$\mathbb{P}\left(N(t^*)=0\right)=e^{-\frac{\lambda\,y}{c}}.$$
In view of the discussion above, it holds that \begin{equation}\mathbb{P}\left(X(\tau)=x\big\lvert X(0)=x,\; Y(0)=y,\;D(0)=d_3\right)=e^{-\frac{\lambda\,y}{c}}.\label{spve}\end{equation}
Thus, the exit point distribution $u_3(x,y\mathord{;}z)$ exhibits a singular component in the point $z=x$, that is in the point of the boundary having the same abscissa as the starting point. On the other hand, if at least a change of direction occurs, the distribution of the exit point of $\big(X(t),Y(t)\big)$ from $\Gamma$ is continuous, and spreads over the whole unbounded boundary of the strip. We denote by $u_3^*(x,y\mathord{;}z)$ this continuous component of the distribution, that is
\begin{align}u_3^*(x,y\mathord{;}z)\;\mathrm{d}z=\mathbb{P}\Big(X(\tau)\in\mathrm{d}z,\,Y(\tau)=0\,&\Big\lvert X(0)=x,\,Y(0)=y,\nonumber\\&D(0)=d_3,\,N(t^*)>0\Big).\nonumber\end{align}
By taking into account formula (\ref{spve}), it is clear that the exit point distribution and its continuous component are related by the equation
\begin{equation}u_3(x,y\mathord{;}z)=e^{-\frac{\lambda\,y}{c}}\,\delta\big(z-x\big)+\left(1-e^{-\frac{\lambda\,y}{c}}\right) u_3^*(x,y\mathord{;}z).\label{singularexplicit}\end{equation}
We now define the Fourier transform
$$\widetilde{u}_3^*(\alpha,\,y\mathord{;}z)=\int_{-\infty}^{+\infty}e^{i\alpha x}\;u_3^*(x,y\mathord{;}z)\,\mathrm{d}x,\qquad 0<y<L,\;\;\alpha\in\mathbb{R}.$$
By taking the Fourier transform of both sides of equation (\ref{singularexplicit}), we obtain that
\begin{equation}\widetilde{u}_3(x,y\mathord{;}z)=e^{-\frac{\lambda\,y}{c}+i\alpha z}+\left(1-e^{-\frac{\lambda\,y}{c}}\right) \widetilde{u}_3^*(x,y\mathord{;}z).\label{singularexplicitfourier}\end{equation}
We now claim that the asymptotic behaviour (\ref{asym3}) of the function $\widetilde{u}_3(x,y\mathord{;}z)$ follows from the fact that the singular component of the distribution is included in $\widetilde{u}_3(x,y\mathord{;}z)$. Hence, for $j=3$, the inverse Fourier transform must be interpreted in the distributional sense. However, a rigorous representation of the inverse Fourier transform can be obtained by restricting the analysis to the continuous component of the exit point distribution. In particular, formula (\ref{singularexplicitfourier}) implies that
\begin{equation}\widetilde{u}_3^*(\alpha,y\mathord{;}z)=\frac{\widetilde{u}_3(\alpha,y\mathord{;}z)-e^{-\frac{\lambda\,y}{c}+i\alpha z}}{1-e^{-\frac{\lambda\,y}{c}}}.\label{singularexplicitfouriercontinuous}\end{equation}
where the explicit expression of  $\widetilde{u}_3(\alpha,\,y)$ is given by (\ref{uhat3solu}). It is now clear that the inverse Fourier transform 
\begin{equation}\label{ustarift}u_3^*(x,y\mathord{;}z)=\frac{1}{2\pi}\int_{-\infty}^{\infty}e^{-i\alpha x}\;\widetilde{u}_3^*(\alpha,y\mathord{;}z)\,\mathrm{d}\alpha,\qquad 0<y<L,\;\;x\in\mathbb{R}\end{equation} is well-defined, as the integral is convergent. In particular, the convergence can be proved by employing the Dirichlet's test.\\
\noindent The inverse Fourier transforms (\ref{smville}) and (\ref{ustarift}) do not appear to admit a representation in terms of elementary functions. Nevertheless, in view on the results presented so far, an integral representation is immediately obtained. In particular, by substituting formulas  (\ref{uhat1solu}) and (\ref{uhat3solu}) into (\ref{smville}) and (\ref{ustarift}), and performing the change of variables $w=\frac{\alpha\,c}{\sqrt{\lambda^2+\alpha^2c^2}}$, we can write that

\begin{equation}u_1(x,y\mathord{;}z)=\frac{2\lambda}{c\,\pi}\int_0^1\frac{\cos\left(\frac{(z-x)\,\lambda\,w}{c\,\sqrt{1-w^2}}\right)\;\sinh\left(\frac{\lambda (L-y)}{c}\,w\right)}{(1+w)^2\,e^{\frac{\lambda L}{c}\,w}-(1-w)^2\,e^{-\frac{\lambda L}{c}\,w}}\,\frac{dw}{\sqrt{1-w^2}}\label{inegrep}\end{equation}
and
\begin{align}
u_3(x&,y\mathord{;}z)=e^{-\frac{\lambda\,y}{c}}\,\delta\big(z-x\big)+\nonumber\\
&+\frac{\lambda}{c\,\pi}\int_0^1\cos\left(\frac{(z-x)\,\lambda\,w}{c\,\sqrt{1-w^2}}\right)\cdot\nonumber\\
&\qquad\cdot\left[\frac{(1+w)^2\,e^{\frac{\lambda (L-y)}{c}\,w}-(1-w)^2\,e^{-\frac{\lambda (L-y)}{c}\,w}}{(1+w)^2\,e^{\frac{\lambda L}{c}\,w}-(1-w)^2\,e^{-\frac{\lambda L}{c}\,w}}-e^{-\frac{\lambda y}{c}}\right]\,\frac{dw}{(1-w^2)^{3/2}}\nonumber
\end{align}
where the continuous component (\ref{singularexplicitfouriercontinuous}) of $u_3(x,y\mathord{;}z)$ reads
\begin{align}
u^*_3(x&,y\mathord{;}z)\nonumber\\
&=\frac{\lambda}{c\,\pi\left(1-e^{-\frac{\lambda\,y}{c}}\right)}\int_0^1\cos\left(\frac{(z-x)\,\lambda\,w}{c\,\sqrt{1-w^2}}\right)\cdot\nonumber\\
&\qquad\cdot\left[\frac{(1+w)^2\,e^{\frac{\lambda (L-y)}{c}\,w}-(1-w)^2\,e^{-\frac{\lambda (L-y)}{c}\,w}}{(1+w)^2\,e^{\frac{\lambda L}{c}\,w}-(1-w)^2\,e^{-\frac{\lambda L}{c}\,w}}-e^{-\frac{\lambda y}{c}}\right]\,\frac{dw}{(1-w^2)^{3/2}}\nonumber
\end{align}
Moreover, by taking into account formula (\ref{algebrain}), we have that
\begin{align}
u_0(x,y\mathord{;}z)=\frac{\lambda}{c\,\pi}\int_0^1&\frac{\sqrt{1-w^2}\,\cos\left(\frac{(z-x)\,\lambda\,w}{c\,\sqrt{1-w^2}}\right)+w\,\sin\left(\frac{(z-x)\,\lambda\,w}{c\,\sqrt{1-w^2}}\right)}{1-w^2}\cdot\nonumber\\
&\qquad\cdot\frac{(1+w)\,e^{\frac{\lambda (L-y)}{c}\,w}-(1-w)\,e^{-\frac{\lambda (L-y)}{c}\,w}}{(1+w)^2\,e^{\frac{\lambda L}{c}\,w}-(1-w)^2\,e^{-\frac{\lambda L}{c}\,w}}\;dw\nonumber
\end{align}
and
\begin{align}
u_2(x,y\mathord{;}z)=\frac{\lambda}{c\,\pi}\int_0^1&\frac{\sqrt{1-w^2}\,\cos\left(\frac{(z-x)\,\lambda\,w}{c\,\sqrt{1-w^2}}\right)-w\,\sin\left(\frac{(z-x)\,\lambda\,w}{c\,\sqrt{1-w^2}}\right)}{1-w^2}\cdot\nonumber\\
&\qquad\cdot\frac{(1+w)\,e^{\frac{\lambda (L-y)}{c}\,w}-(1-w)\,e^{-\frac{\lambda (L-y)}{c}\,w}}{(1+w)^2\,e^{\frac{\lambda L}{c}\,w}-(1-w)^2\,e^{-\frac{\lambda L}{c}\,w}}\;dw\nonumber
\end{align}
\noindent The formulas above provide explicit integral representations for the probability density functions $u_j(x,y\mathord{;}z),\; j=0,1,2,3$, of the exit point $z$ on the lower boundary of the strip $\Gamma$. Although such integrals do not seem to be solvable analytically, numerical integration can be performed in order to compute the density functions. In this way, we are able to obtain the graphical representation of the densities shown in Figure \ref{fig:densities}, which offers an insightful view of their behaviour.
\begin{figure}[h]\label{fig:densities}
\centering
\includegraphics[scale=0.265]{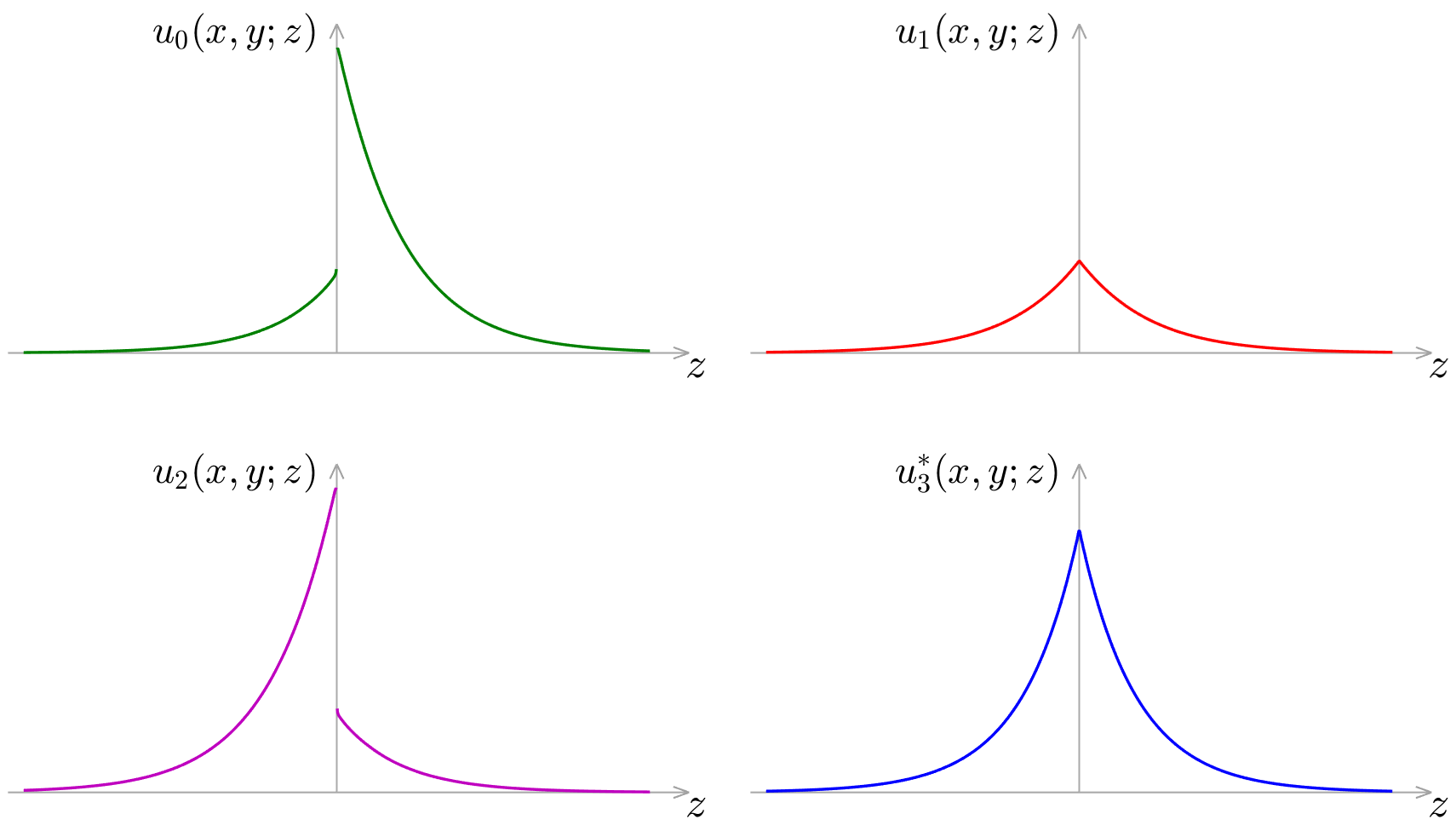}
\caption{Graphical representation of the density functions $u_j(x,y\mathord{;}z),\,j=0,1,2$, and $u_3^*(x,y\mathord{;}z)$. The density functions have been computed by numerically inverting the Fourier transforms obtained in Theorem \ref{tildeuthm} for $x=0,\;y=\frac{1}{2},\;L=1,\;c=5,\;\lambda=10$. When the process $\big(X(t),Y(t)\big)$ starts moving vertically at time $t=0$, the distribution of the exit point $z$ is symmetric around the abscissa of the starting point, as suggested by the densities $u_1(x,y\mathord{;}z)$ and $u_3^*(x,y\mathord{;}z)$. When the particle starts moving downwards with direction $d_3$, the probability of exiting the strip from the lower boundary is higher compared to the case in which the initial motion is upwards. For this reason, the integral of density $u_3^*(x,y\mathord{;}z)$ is higher than that of $u_1(x,y\mathord{;}z)$, as we will also later confirm numerically in the paper. When the process $\big(X(t),Y(t)\big)$ starts moving horizontally at time $t=0$, the density of the exit point is asymmetric and exhibits a jump discontinuity at the abscissa of the starting. In particular, if the initial direction is leftwards (rightwards) the probability of exiting the strip to the left (right) of the initial abscissa is greater.}
\end{figure}

Having determined the exact distribution of the exit point of the process $\big(X(t),Y(t)\big)$ from the lower boundary of the infinite strip $\Gamma$, we are now interested in calculating the probability of the process exiting the strip through the lower boundary. For this purpose, we define, for $j=0,1,2,3$ and $(x,y)\in\Gamma$, the functions \begin{equation}\label{nikol}p_j(x,y)=\mathbb{P}\Big(Y(\tau)=0\;\Big\lvert X(0)=x,\,Y(0)=y,\,D(0)=d_j\Big).\end{equation} Similarly, we define the overall probability of exiting though the lower boundary as 
$$p(x,y)=\mathbb{P}\Big(Y(\tau)=0\;\Big\lvert X(0)=x,\,Y(0)=y\Big).$$ Clearly, it holds that $p(x,z)=\frac{1}{4}\sum_{j=0}^3p_j(x,y)$. By means of standard methods, it can be verified that 
\begin{equation}\label{IMErc}
\begin{cases}
p_0(x,y)=p_0(x+c\,\Delta t,y)\,(1-\lambda\,\Delta t)+p_1(x,y)\,\frac{\lambda\,\Delta t}{2}+p_3(x,y)\,\frac{\lambda\,\Delta t}{2} + o(\Delta t)\\
p_1(x,y)=p_1(x,y+c\,\Delta t)\,(1-\lambda\,\Delta t)+p_0(x,y)\,\frac{\lambda\,\Delta t}{2}+p_2(x,y)\,\frac{\lambda\,\Delta t}{2} + o(\Delta t)\\
p_2(x,y)=p_2(x-c\,\Delta t,y)\,(1-\lambda\,\Delta t)+p_1(x,y)\,\frac{\lambda\,\Delta t}{2}+p_3(x,y)\,\frac{\lambda\,\Delta t}{2} + o(\Delta t)\\
p_3(x,y)=p_3(x,y-c\,\Delta t)\,(1-\lambda\,\Delta t)+p_0(x,y)\,\frac{\lambda\,\Delta t}{2}+p_2(x,y)\,\frac{\lambda\,\Delta t}{2} + o(\Delta t).
\end{cases}\end{equation}
Therefore, by taking a first-order Taylor expansion of the system (\ref{IMErc}), we can write that
\begin{equation}\label{bivpdsys}
\begin{cases}
\dfrac{\partial p_0}{\partial x}=-\dfrac{\lambda}{2c}\,\left[p_1+p_3-2p_0\right]\\[0.75em]
\dfrac{\partial p_1}{\partial y}=-\dfrac{\lambda}{2c}\,\left[p_0+p_2-2p_1\right]\\[0.75em]
\dfrac{\partial p_2}{\partial x}=\dfrac{\lambda}{2c}\,\left[p_1+p_3-2p_2\right]\\[0.75em]
\dfrac{\partial p_3}{\partial y}=\dfrac{\lambda}{2c}\,\left[p_0+p_2-2p_3\right].
\end{cases}\end{equation}
It is interesting to observe that the system (\ref{bivpdsys}) can also be obtained by using formula (\ref{bivudsys}) and taking into account that \begin{equation}p_j(x,y)=\int_{-\infty}^{+\infty}u_j(x,y\mathord{;}x).\label{doubleintegral}\end{equation}
Observe that, in principle, formula (\ref{doubleintegral}) provides an immediate representation for the probabilities $p_j(x,y),\;j=0,1,2,3$. However, taking into account that we have obtained a representation for the densities $u_j(x,y),\;j=0,1,2,3$ in integral form, it turns out that formula (\ref{doubleintegral}) yields the desired probabilities in the form of double integrals. The reader can readily verify that such double integrals are analytically intractable. Interestingly, by employing the system (\ref{bivpdsys}), we are able to obtain a representation for $p_j(x,y),\;j=0,1,2,3$, in an elementary form. Moreover, by numerically computing the double integral (\ref{doubleintegral}), we shall prove that such elementary form is consistent with the Fourier transforms presented in Theorem \ref{tildeuthm}.\\
In order to solve the system (\ref{bivpdsys}), we exploit a simplification which emerges as a consequence of the geometrical structure of the infinite strip $\Gamma$. Indeed, since the strip is horizontally unbounded, one can move the starting point horizontally while leaving the geometrical structure of the problem unaltered. In other words, the fact that $\Gamma$ is an infinite horizontal strip permits us to conclude that the functions $p_j(x,y),\;j=0,1,2,3$, do not depend on $x$. This fact follows trivially from the fact that the horizontal strip $\Gamma$ is invariant under horizontal translations. For a formal proof, one needs to consider the horizontally translated process $\big(\widetilde{X}(t), Y(t)\big)$ with $\widetilde{X}(t):=X(t)-x$. Clearly, $\big(\widetilde{X}(t), Y(t)\big)$ is a finite-velocity random motion with orthogonal directions and initial position $(0,y)$. Denoting by $\tau$ and $\widetilde{\tau}$ the exit times of $\big(X(t), Y(t)\big)$ and  $\big(\widetilde{X}(t), Y(t)\big)$ from the strip $\Gamma$ respectively, it is immediate to verify that $\tau=\widetilde{\tau}$. This implies that, for $j=0,1,2,3$,
\begin{align}
p_j(x,y)=&\mathbb{P}\Big(Y(\tau)=0\;\Big\lvert X(0)=x,\,Y(0)=y,\,D(0)=d_j\Big)\nonumber\\
=&\mathbb{P}\Big(Y(\tau)=0\;\Big\lvert \widetilde{X}(0)=0,\,Y(0)=y,\,D(0)=d_j\Big)\nonumber\\
=&\mathbb{P}\Big(Y(\widetilde{\tau})=0\;\Big\lvert \widetilde{X}(0)=0,\,Y(0)=y,\,D(0)=d_j\Big)\nonumber\\
=&p_j(0,y)\label{academi}
\end{align}
Since the relationship (\ref{academi}) holds for all $x\in\mathbb{R}$, the function $p_j(x,y)$ must be independent of $x$. We now observe that, since the probabilities $p_j(x,y),\;j=0,1,2,3,$ do not depend on $x$, we have that $\frac{\partial p_0}{\partial x}=\frac{\partial p_2}{\partial x}=0$. Thus, the first and the third equations of the system (\ref{bivpdsys}) imply that \begin{equation}\label{iside0}p_0(x,y)=p_2(x,y)=\frac{p_1(x,y)+p_3(x,y)}{2}.\end{equation} The second and fourth equations yield
\begin{equation}\label{reducedbivpdsys}
\begin{cases}
\dfrac{\mathrm{d} p_1}{\mathrm{d} y}=\dfrac{\lambda}{2c}\,\left[p_1-p_3\right]\\[0.75em]
\dfrac{\mathrm{d} p_3}{\mathrm{d} y}=\dfrac{\lambda}{2c}\,\left[p_1-p_3\right].
\end{cases}\end{equation}
Hence, we have reduced the fourth-order system (\ref{bivpdsys}) into the second-order system (\ref{reducedbivpdsys}). Moreover, it can be verified that the system (\ref{reducedbivpdsys}) is subject to the boundary conditions
\begin{equation}p_1(x,L)=0\quad\text{and}\quad p_3(x,0)=1,\qquad \forall x\in\mathbb{R}.\label{IMER3}\end{equation}
Thus, we are able to prove the following result.
\begin{thm}\label{thm:p}It holds that
\begin{equation}p_1(x,y)=\frac{\lambda\,(L-y)}{\lambda L+2c},\qquad p_3(x,y)=\frac{2c+\lambda\,(L-y)}{\lambda L+2c}\label{iside1}\end{equation}
and
\begin{equation}p(x,y)=p_0(x,y)=p_2(x,y)=\frac{c+\lambda\,(L-y)}{\lambda L+2c}.\label{iside2}\end{equation}\end{thm}
\begin{proof}
The system (\ref{reducedbivpdsys}) can be expressed as
\begin{equation}\frac{\mathrm{d}}{\mathrm{d} y}\begin{pmatrix}p_1(x,y)\\p_3(x,y)\end{pmatrix}=A\cdot\begin{pmatrix}p_1(x,y)\\p_3(x,y)\end{pmatrix}\label{matrixformp}\end{equation} with $A=\begin{pmatrix}\frac{\lambda}{2c} & -\frac{\lambda}{2c}\\[0.5em]\frac{\lambda}{2c} & -\frac{\lambda}{2c}\end{pmatrix}.$
Since the system (\ref{matrixformp}) has the same form as (\ref{matrixformu}) with $c$ replaced by $2c$, the solution can be achieved as in Theorem \ref{thm1}. Thus, we can write that
$$p_1(x,y)=\kappa_0+(\kappa_0-\kappa_1)\frac{\lambda y}{2c},\qquad p_3(x,y)=\kappa_1+(\kappa_0-\kappa_1)\frac{\lambda y}{2c}$$ where $\kappa_0$ and $\kappa_1$ are arbitrary constants. By imposing the boundary conditions (\ref{IMER3}) we obtain that
$$\kappa_0=\frac{\lambda L}{\lambda L+2c},\qquad\qquad\kappa_1=1$$ which yields formula (\ref{iside1}). Equation (\ref{iside2}) finally follows by the identity (\ref{iside0}).
\end{proof}
Theorem \ref{thm:p} provides an explicit representation of the exit probabilities (\ref{nikol}). These probabilities depend linearly on the ordinate $y$ of the starting point of the process $\big(X(t),Y(t)\big)$. This fact is noteworthy, since an equivalent representation of the probabilities $p_j(x,y),\;j=0,1,2,3$, is given by the integral formula (\ref{doubleintegral}). In particular, we recall that formula (\ref{doubleintegral}) involves double integrals, since the integrand functions themselves admit the integral representations described in (\ref{inegrep}) and the subsequent formulas. To verify the consistency between the linear expressions derived in Theorem \ref{thm:p} and the integral representation (\ref{doubleintegral}), we carry out numerical experiments by evaluating (\ref{doubleintegral}) via numerical integration. The results of the numerical experiments are illustrated in Figure \ref{fig:intdens} and confirm the consistency of the results presented so far.

\begin{figure}[h!]\label{fig:intdens}
\centering
\includegraphics[scale=0.7]{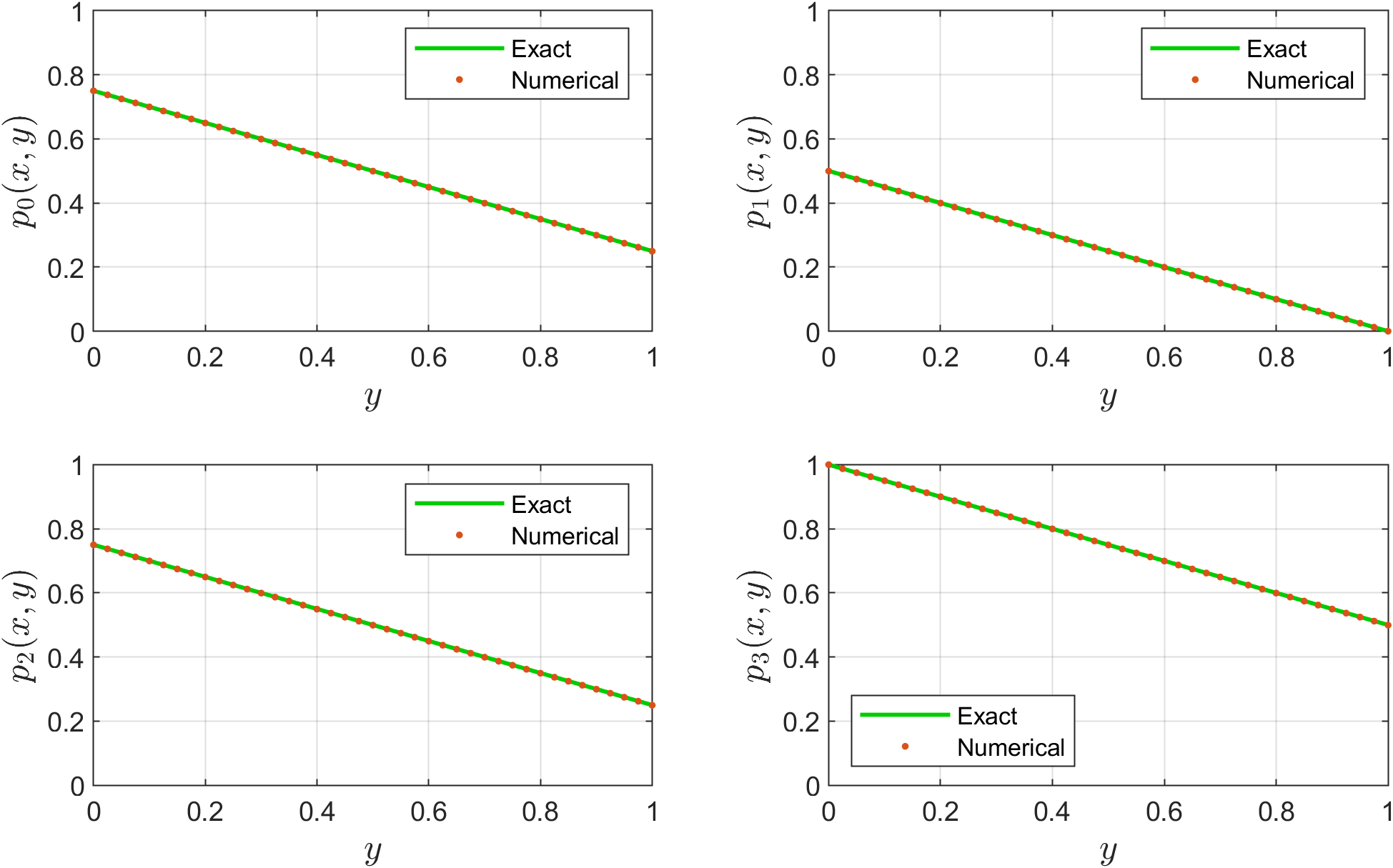}
\caption{Comparison of the exact probabilities of exiting the strip $\Gamma$ through the lower boundary and the corresponding numerical probabilities obtained by computing the double integrals (\ref{doubleintegral}) for different values of the initial abscissa $x$. For the computation of the double integrals, the inner integrals were calculated by using adaptive quadrature algorithms, while the outer integrals were computed via the trapezoidal rule. All the probabilities were obtained for $L=1$, $\lambda=10$ and $c=5$.}
\end{figure}

We conclude our study of the process $\big(X(t),Y(t)\big)$ in the infinite strip $\Gamma$ by investigating the mean exit time from the strip as a function of the starting point $(x,y)$. Thus, for $j=0,1,2,3$, we consider the functions
\begin{equation}\label{bbhb}h_j(x,y)=\mathbb{E}\Big[\tau\;\Big\lvert\;X(0)=x,\,Y(0)=y,\,D(0)=d_j\Big],\qquad (x,y)\in\Gamma.\end{equation}
Similarly, by removing the conditioning with respect to the initial direction, we define the function
$$h(x,y)=\mathbb{E}\Big[\tau\;\Big\lvert\;X(0)=x,\,Y(0)=y\Big],\qquad (x,y)\in\Gamma.$$
Clearly, we have that 
\begin{equation}h(x,y)=\frac{1}{4}\sum_{j=0}^3 h_j(x,y).\label{unchg}\end{equation}
By using the methods described in Section \ref{sec:base}, it can be verified that the expected values (\ref{bbhb}) satisfy the following system of equations
\begin{equation}\label{IMErcod}
\begin{cases}
h_0(x,y)=\Delta t+h_0(x+c\,\Delta t,y)\,(1-\lambda\,\Delta t)+\big[h_1(x,y)+h_3(x,y)\big]\,\frac{\lambda\,\Delta t}{2} + o(\Delta t)\\[0.35em]
h_1(x,y)=\Delta t+h_1(x,y+c\,\Delta t)\,(1-\lambda\,\Delta t)+\big[h_0(x,y)+h_2(x,y)\big]\,\frac{\lambda\,\Delta t}{2} + o(\Delta t)\\[0.35em]
h_2(x,y)=\Delta t+h_2(x-c\,\Delta t,y)\,(1-\lambda\,\Delta t)+\big[h_1(x,y)+h_3(x,y)\big]\,\frac{\lambda\,\Delta t}{2} + o(\Delta t)\\[0.35em]
h_3(x,y)=\Delta t+h_3(x,y-c\,\Delta t)\,(1-\lambda\,\Delta t)+\big[h_0(x,y)+h_2(x,y)\big]\,\frac{\lambda\,\Delta t}{2} + o(\Delta t).
\end{cases}\end{equation}
A first-order expansion of the system (\ref{IMErcod}) leads to the linear system of partial differential equations
\begin{equation}\label{bivhjdsys}
\begin{cases}
\dfrac{\partial h_0}{\partial x}=-\dfrac{\lambda}{2c}\,\left[h_1+h_3-2h_0\right]-\dfrac{1}{c}\\[0.75em]
\dfrac{\partial h_1}{\partial y}=-\dfrac{\lambda}{2c}\,\left[h_0+h_2-2h_1\right]-\dfrac{1}{c}\\[0.75em]
\dfrac{\partial h_2}{\partial x}=\dfrac{\lambda}{2c}\,\left[h_1+h_3-2h_2\right]+\dfrac{1}{c}\\[0.75em]
\dfrac{\partial h_3}{\partial y}=\dfrac{\lambda}{2c}\,\left[h_0+h_2-2h_3\right]+\dfrac{1}{c}.
\end{cases}\end{equation}
It is interesting to observe that the system (\ref{bivhjdsys}) implies that the function $h(x,y)$ satisfies the partial differential equation
\begin{equation}\label{hpoispde}\Bigg(\Delta-\frac{c^2}{\lambda^2}\,\frac{\partial^4}{\partial x^2\,\partial y^2}\Bigg)h=-\frac{2\lambda}{c^2}\end{equation}
where $\Delta$ is the bivariate Laplace operator. Clearly, equation (\ref{hpoispde}) is also satisfied by the functions $h_j(x,y),\;j=0,1,2,3$. Moreover, we emphasize that equation (\ref{hpoispde}) represents a finite-velocity extension of the classical bivariate Poisson equation arising in the study of Brownian motion. Indeed, by taking the hydrodynamic limit for $\lambda,c\to+\infty$ with $\frac{\lambda}{c^2}\to1$, equation (\ref{hpoispde}) reduces to the classical Poisson equation $$\Delta h = -2$$ which governs the mean exit time of a bivariate Brownian motion. This is consistent with the limiting distribution of the finite-velocity process $\big(X(t),Y(t)\big)$, which converges in distribution to a Brownian motion in the hydrodynamic limit. Similarly to the finite-velocity Laplace equation (\ref{extendedlaplacian}), the extended Poisson equation (\ref{hpoispde}) is independent of the geometry of the set in which the mean exit time is studied. On the other hand, the solution $h(x,y)$ depends on the shape of the considered set $\Gamma$.\\
In the special case in which $\Gamma$ is the horizontal infinite strip, the problem undergoes a remarkable simplification. In particular, similarly to the exit probability (\ref{nikol}), it can be verified that the mean exit time (\ref{bbhb}) from the infinite strip does not depend on the abscissa $x$ of the starting point. This can be verified by means of the same procedure described in equation (\ref{academi}). Therefore, the first and third differential equations of the system (\ref{bivhjdsys}) reduce to the algebraic equation
\begin{equation}\label{bivhjdsysalg}
h_0(x,y)=h_2(x,y)=\frac{h_1(x,y)+h_3(x,y)}{2}+\dfrac{1}{\lambda}.
\end{equation}
Moreover, the second and fourth equations of the system can be expressed as
\begin{equation}\label{bivhjdsysredu}
\begin{cases}
\dfrac{\mathrm{d} h_1}{\mathrm{d} y}=-\dfrac{\lambda}{2c}\,\left[h_3-h_1\right]-\dfrac{2}{c}\\[0.75em]
\dfrac{\mathrm{d} h_3}{\mathrm{d} y}=-\dfrac{\lambda}{2c}\,\left[h_3-h_1\right]+\dfrac{2}{c}.
\end{cases}\end{equation}
To solve the system (\ref{bivhjdsysredu}), we observe that the following boundary conditions must hold:
\begin{equation}\label{tashsi}h_1(x,L)=h_3(x,0)=0.\end{equation}
Thus, the mean exit time of $\big(X(t),Y(t)\big)$ from the horizontal strip $\Gamma$ is given in the following theorem.
\begin{thm}
For all $x\in\mathbb{R},\;y\in[0,L]$, it holds that
\begin{equation}\label{got1bbb}h_1(x,y)=\frac{\lambda\;y\,(L-y)}{c^2}+\frac{2\,(L-y)}{c}\end{equation}
and
\begin{equation}\label{got2bbb}h_3(x,y)=\frac{\lambda\;y\,(L-y)}{c^2}+\frac{2y}{c}.\end{equation}
Consequently, we have that
\begin{equation}\label{got3bbb}h_0(x,y)=h_2(x,y)=\frac{\lambda\;y\,(L-y)}{c^2}+\frac{\lambda L +c}{\lambda\,c}\end{equation}
and \begin{equation}\label{got4bbb}h(x,y)=\frac{\lambda\;y\,(L-y)}{c^2}+\frac{2\lambda L +c}{2\lambda\,c}.\end{equation}
\end{thm}
\begin{proof}
We start by observing that the system (\ref{bivhjdsysredu}) can be expressed in the form
\begin{equation}\frac{\mathrm{d}}{\mathrm{d} y}\begin{pmatrix}h_1(x,y)\\[0.3em]h_3(x,y)\end{pmatrix}=A\cdot\begin{pmatrix}h_1(x,y)\\[0.3em]h_3(x,y)\end{pmatrix}+\begin{pmatrix}-\frac{2}{c}\\[0.3em]\;\frac{2}{c}\end{pmatrix}\label{matrixformhbiv}\end{equation} with $A=\begin{pmatrix}\frac{\lambda}{2c} & -\frac{\lambda}{2c}\\[0.5em]\frac{\lambda}{2c} & -\frac{\lambda}{2c}\end{pmatrix}$. By using formula (\ref{hgeneralsol}) as in Theorem \ref{thm:fring}, the general solution to (\ref{matrixformhbiv}) reads
$$h_1(x,y)=\kappa_0+(\lambda\kappa_0-\lambda\kappa_1-4)\,\frac{y}{2c}-\frac{\lambda y^2}{c^2}$$
and
$$h_3(x,y)=\kappa_1+(\lambda\kappa_0-\lambda\kappa_1+4)\,\frac{y}{2c}-\frac{\lambda y^2}{c^2}.$$
By imposing the boundary conditions (\ref{tashsi}), we obtain that
$$\kappa_0=\frac{2L}{c},\qquad\qquad \kappa_1=0$$
which yield formulas (\ref{got1bbb}) and (\ref{got2bbb}). Formulas (\ref{got3bbb}) and (\ref{got4bbb}) follow immediately by using (\ref{unchg}) and (\ref{bivhjdsysalg}).
\end{proof}

\bibliographystyle{plain}
\nocite{*}
\bibliography{bibliography}

@article{cinqueasymm,
title = {On the exact distributions of the maximum of the asymmetric telegraph process},
journal = {Stochastic Processes and their Applications},
volume = {142},
pages = {601-633},
year = {2021},
issn = {0304-4149},
author = {Fabrizio Cinque and Enzo Orsingher}
}

@article{cinquemax,
  title={On the distribution of the maximum of the telegraph process},
  author={Cinque, Fabrizio and Orsingher, Enzo},
  journal={Theory of Probability and Mathematical Statistics},
  volume={102},
  pages={73--95},
  year={2020}
}

@book{port2012brownian,
  title={Brownian motion and classical potential theory},
  author={Port, Sidney and Stone, Charles},
  year={2012},
  publisher={Elsevier},
  address={New York}
}

@book{doob1984classical,
  title={Classical potential theory and its probabilistic counterpart},
  author={Doob, Joseph L and others},
  volume={19},
  year={1984},
  publisher={Springer}
}

@book{karatzas2014brownian,
  title={Brownian motion and stochastic calculus},
  author={Karatzas, Ioannis and Shreve, Steven},
  year={2014},
  publisher={Springer}
}

@article{pedicone,
  title={On the distribution of the telegraph meander and its properties},
  author={Pedicone, Andrea and Orsingher, Enzo},
  journal={Stochastic Processes and their Applications},
  pages={104887},
  year={2026},
}

@article{vortex,
  title={Planar random motions in a vortex},
  author={Orsingher, Enzo and Marchione, Manfred Marvin},
  journal={Journal of Theoretical Probability},
  volume={38},
  number={1},
  pages={4},
  year={2025},
}

@article{orsingher1990,
  title={Probability law, flow function, maximum distribution of wave-governed random motions and their connections with Kirchoff's laws},
  author={Orsingher, Enzo},
  journal={Stochastic Processes and their Applications},
  volume={34},
  number={1},
  pages={49--66},
  year={1990},
  publisher={Elsevier}
}

@article{bogachev,
  title={Occupation time distributions for the telegraph process},
  author={Bogachev, Leonid and Ratanov, Nikita},
  journal={Stochastic Processes and their Applications},
  volume={121},
  number={8},
  pages={1816--1844},
  year={2011}
}

@article{foong,
  title={Properties of the telegrapher's random process with or without a trap},
  author={Foong, SK and Kanno, S},
  journal={Stochastic processes and their applications},
  volume={53},
  number={1},
  pages={147--173},
  year={1994}
}

@article{Ratanov,
author = {Nikita Ratanov},
title = {A jump telegraph model for option pricing},
journal = {Quantitative Finance},
volume = {7},
number = {5},
pages = {575--583},
year = {2007}}

@article{dicrpellerey,
  title={On prices' evolutions based on geometric telegrapher's process},
  author={Di Crescenzo, Antonio and Pellerey, Franco},
  journal={Applied Stochastic Models in Business and Industry},
  volume={18},
  number={2},
  pages={171--184},
  year={2002}}

@article{dicrminimal,
  title={Exact transient analysis of a planar random motion with three directions},
  author={Di Crescenzo, Antonio},
  journal={Stochastics: An International Journal of Probability and Stochastic Processes},
  volume={72},
  number={3-4},
  pages={175--189},
  year={2002},
  publisher={Taylor \& Francis}
}

@article{orsingher2000exact,
  title={Exact joint distribution in a model of planar random motion},
  author={Orsingher, Enzo},
  journal={Stochastics: An International Journal of Probability and Stochastic Processes},
  volume={69},
  number={1-2},
  pages={1--10},
  year={2000}
}

@article{beghinnieddu,
  title={Probabilistic analysis of the telegrapher's process with drift by means of relativistic transformations},
  author={Beghin, Luisa and Nieddu, Luciano and Orsingher, Enzo},
  journal={International Journal of Stochastic Analysis},
  volume={14},
  number={1},
  pages={11--25},
  year={2001}
}

@article{goldstein1951,
  title={On diffusion by discontinuous movements, and on the telegraph equation},
  author={Goldstein, Sidney},
  journal={The Quarterly Journal of Mechanics and Applied Mathematics},
  volume={4},
  number={2},
  pages={129--156},
  year={1951}
}

@article{cane1975,
  title={Diffusion models with relativity effects},
  author={Cane, Violet R},
  journal={Journal of Applied Probability},
  volume={12},
  number={S1},
  pages={263--273},
  year={1975}
}

@article{kac1974stochastic,
  title={A stochastic model related to the telegrapher's equation},
  author={Kac, Mark},
  journal={The Rocky Mountain Journal of Mathematics},
  volume={4},
  number={3},
  pages={497--509},
  year={1974},
  publisher={JSTOR}
}

@article{MOcorr,
  title={On a planar random motion with asymptotically correlated components},
  author={Marchione, Manfred Marvin and Orsingher, Enzo},
  journal={Journal of Statistical Physics},
  volume={191},
  number={10},
  pages={121},
  year={2024}
}
\end{document}